\documentclass[letter, 11pt]{article}

\usepackage{natbib}
\usepackage{amssymb,amsmath,xcolor,graphicx,xspace,colortbl,rotating} %
\usepackage[raggedrightboxes]{ragged2e} 
\usepackage{subfigure}
\usepackage{appendix}  
\usepackage{bm} 
\usepackage{cancel} 
\usepackage{boxedminipage}  
\usepackage{color}  
\usepackage{endnotes}  
\usepackage{ragged2e}  
\usepackage[onehalfspacing]{setspace}  
\usepackage{tabulary}  
\usepackage{textcomp}  
\usepackage{varioref}  
\usepackage{graphicx}
\graphicspath{ {./images/} }
 
\usepackage[margin=1in]{geometry}
\newtheorem {theorem}{Theorem}

\newtheorem {corollary}{Corollary}

\newtheorem {definition}{Definition}
\newtheorem {example}{Example}

\newtheorem {lemma}{Lemma}

\newtheorem {proposition}{Proposition}

\newenvironment {proof}[1][Proof]{\noindent \textbf {#1.} }{\ \rule {0.5em}{0.5em}}
\newcommand{\R}{\mathbb{R}}
\newcommand{\N}{\mathbb{N}}
\newcommand{\E}{\mathbb{E}}

\renewcommand{\alph}{\alpha, [a,b]}

\makeatletter
\let\@fnsymbol\@arabic
\makeatother
\begin{document}

\title{The Family of Alpha,[a,b] Stochastic Orders: \\ Risk vs. Expected Value}

\author{Bar Light\protect\thanks{Graduate School of Business,
Stanford University, Stanford, CA 94305, USA. e-mail: \textsf{barl@stanford.edu}\ } ~ and  Andres Perlroth\protect\thanks{  Google Research. e-mail: \textsf{perlroth@google.com }. Most of the research for this paper was performed while this author was a student at Stanford GSB.} } 
\maketitle

\thispagestyle{empty}

 \noindent \noindent \textsc{Abstract}:
\begin{quote}
	In this paper we provide a novel family of stochastic orders that generalizes second order stochastic dominance, which we call the $\alpha,[a,b]$-concave  stochastic orders. 
	These stochastic orders are generated by a novel set of ``very" concave functions where $\alpha$ parameterizes the degree of concavity. The $\alpha,[a,b]$-concave stochastic orders  allow us to derive novel comparative statics results for important applications in economics that cannot be derived using previous stochastic orders.	In particular, our comparative statics results are useful when an increase in a lottery's riskiness changes the agent's optimal action in the opposite direction to an increase in the lottery's expected value.  For this kind of situation, we provide a tool to determine which of these two forces dominates -- riskiness or expected value. We apply our results in consumption-savings problems, self-protection problems, and in a Bayesian game.  
	
\end{quote}



\newpage 

\section{Introduction}

Stochastic orders are fundamental in the study of decision making under uncertainty and in the study of complex stochastic systems. They have been used in various fields, including economics, finance, operations research, and statistics (for a textbook treatment of stochastic orders and their applications, see \cite{muller2002comparison},  \cite{shaked2007stochastic}, or  \cite{levy2015stochastic}). In this paper we provide a family of stochastic orders that is based on a novel family of utility functions, which allows us to compare two random variables, where  one  random variable has a higher expected value and is also riskier than the other random variable. 

 For instance, consider the following two simple random variables (also called lotteries) $\tilde{Y}$ and $\tilde{X}$ described in Figure \ref{fig:1}. 
 
 \begin{figure}[h]
\includegraphics[width=4cm]{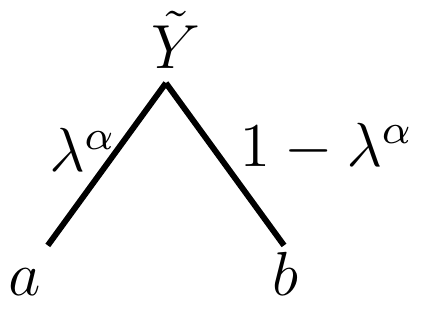} \
\includegraphics[width=4cm]{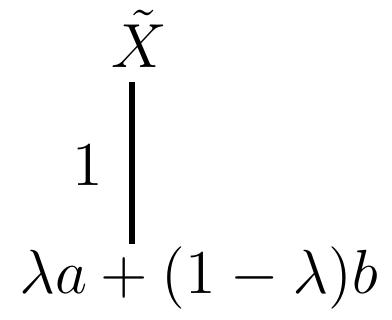}
\centering
\label{fig:1}
\caption{Example 1}
\end{figure}
 
 Lottery $\tilde{Y}$ yields $a$ dollars with probability $\lambda^{\alpha}$ and $b$ dollars with probability $1 -\lambda^{\alpha}$ where $b>a$, $\lambda \in [0,1]$, and $\alpha  \geq 1$. Lottery $\tilde{X}$ yields $\lambda a +\left (1 -\lambda \right )b$ dollars with probability $1$.  If $\alpha $ is not very high, it is reasonable to assume that most risk-averse decision makers would prefer lottery $\tilde{X}$ over lottery $\tilde{Y}$. For example, if $\alpha  =1.152$, $\lambda =0.5$, $a =0$, and $b =1,000 ,000$, then lottery $\tilde{X}$ yields $500,000$ dollars with probability $1$ while lottery $\tilde{Y}$ yields $1,000,000$ dollars with probability $0.55$ and $0$ dollars with probability $0.45$. Lottery $\tilde{Y}$ has a higher expected value ($550 ,000$ dollars) than lottery $\tilde{X}$ but a high probability (a probability of $0.45$) of receiving  $0$ dollars. Thus, in this case, it seems reasonable that most risk-averse decision makers would  prefer lottery $\tilde{X}$ over lottery $\tilde{Y}$. Note that for every $\alpha >1$, lottery $\tilde{Y}$ has a higher expected value and is riskier than lottery $\tilde{X}$. Thus, standard stochastic orders cannot compare the two lotteries. In particular, since the expected value of $\tilde{Y}$ is higher than the expected value of $\tilde{X}$, $\tilde{X}$ does not dominate $\tilde{Y}$ in most popular stochastic orders because these stochastic orders impose a ranking over expectations to determine whether $\tilde{X}$ dominates $\tilde{Y}$. In particular, $\tilde{X}$ does not dominate $\tilde{Y}$ in the second order stochastic dominance (\cite{hadar1969rules} and \cite{rothschild1970increasing}), third order stochastic dominance \citep{whitmore1970third}, higher order stochastic dominance \citep{ekern1980increasing}, decreasing absolute risk aversion stochastic dominance \citep{vickson1977stochastic}, or in the almost second order stochastic dominance \citep{leshno2002preferred}. In Section \ref{sec:2}, however, we show that the stochastic orders provided in this paper that are based on a novel set of risk-averse decision makers can compare $\tilde{X}$ and $\tilde{Y}$.

In this paper we provide a family of stochastic orders indexed by $\alpha,[a,b]$ where $\alpha \geq 1$ and $[a,b]$ is a subset of $\R$, which we call the $\alpha,[a,b]$-concave  stochastic orders. The family of $\alpha,[a,b]$-concave  stochastic orders generalizes second order stochastic dominance (SOSD),\footnote{Recall that $Y$ dominates $X$ in the second order stochastic dominance if $\mathbb{E}[u(Y)] \ge \mathbb{E}[u(X)]$ holds for every  concave and increasing function $u:[a,b] \rightarrow \R$.} which corresponds to the $1,[a,b]$-concave stochastic order.  The main idea of the $\alpha,[a,b]$-concave  stochastic orders is that the inequality  $\mathbb{E}[u(Y)] \ge \mathbb{E}[u(X)]$ is required to hold only for a subset of the concave and increasing functions (and not for all of them) in order to determine that a random variable $Y$ dominates a random variable $X$ in the $\alpha,[a,b]$-concave  stochastic order. In particular, the inequality $\mathbb{E}[u(Y)] \ge \mathbb{E}[u(X)]$ is not required to hold for a function $u$ that is affine or for a function $u$ that is nearly affine in the sense that the elasticity of  $u'$ with respect to $u$ is bounded below by a number that depends on $\alpha$.  This elasticity measures the function's concavity degree in a natural way and relates to the coefficients of prudence and risk aversion (see Section \ref{sec:2} for more details).

An important feature of the $\alpha,[a,b]$-concave  stochastic orders is that for $\alpha>1$, $Y$ dominating $X$ in these orders does not imply that $\E[Y]$ has to be lower than $\E[X]$, nor does it imply the opposite. In Section \ref{sec:2}  we provide examples of random variables $X$ and $Y$ where $X$ has a higher expected value and is riskier than $Y$, and $Y$ dominates $X$ in the $\alpha,[a,b]$-concave  stochastic order. For instance, we show that $\tilde{X}$ dominates $\tilde{Y}$ in the $\alpha,[a,b]$-concave  stochastic order for the example presented in Figure~\ref{fig:1}. Another feature of the $\alpha,[a,b]$-concave stochastic orders is their dependence on the support of the distribution. We show that this dependence is helpful for applications where agents' behavior depends on their wealth level. We illustrate this in a consumption-savings example (see Section \ref{sec: Motivating}). 



For general random variables it is not trivial to check whether a random variable dominates another random variable in the $\alpha,[a,b]$-concave  stochastic orders.  Finding a simple integral condition to characterize stochastic orders that generalize SOSD is impossible or not trivial (see \cite{gollier2018new}). However, we provide a sufficient condition for domination in the $\alpha,[a,b]$-concave  stochastic order that is based on a simple integral inequality (see Section \ref{sec:2}). Similar integral conditions are used to determine whether a random variable dominates another random variable in other popular stochastic orders. The sufficient condition generates a stochastic order that is of independent interest and can be easily used in applications. We partially characterize the maximal generator of this new stochastic order (see Appendix \ref{section:maximal}) for $\alpha=2$.

 To illustrate the usefulness of the family of $\alpha,[a,b]$-concave  stochastic orders, we derive novel comparative statics results in three applications from the economics literature. 
The first application is a consumption-savings problem with labor income uncertainty. It is established in previous literature that a prudent agent  (i.e., an agent whose utility function has a positive third derivative) saves more if the labor income risk increases in the sense of SOSD (see \cite{leland1968saving}). It is also easy to establish that the agent's current savings increase if the labor income's expected present value increases. We do not know of any comparative statics results for the case when both the present value and the risk of future labor income increase. We show that under certain conditions on the agent's marginal utility (the marginal utility must be ``very convex''), an increase in the risk of future labor income together with an increase in the expected present value of future  labor income increase savings. That is,  the precautionary saving motive is stronger than the permanent income motive.

The second application deals with self-protection problems. We consider a standard self-protection problem (e.g., \cite{ehrlich1972market}) where choosing a  higher action is more costly but reduces the probability of a loss. Stochastic orders can be used as a tool to decide whether the level of self-protection should be higher or lower. For a decision maker that makes decisions according to the decision rule implied by the $\alpha,[a,b]$-concave   stochastic order, we provide conditions that imply a change in the level of self-protection. 

In our third application, we show that the $\alpha,[a,b]$-concave stochastic order can be used in a non-cooperative framework as well. We study a Bayesian game which is a variant of the search model studied in \cite{diamond1982aggregate}  and in  \cite{milgrom1990rationalizability}. In this game, there are two players that exert a costly effort to achieve a match, and the probability of a match occurring depends on the effort exerted by both. 
We analyze how different beliefs affect the equilibrium probability of matching.

Our $\alpha,[a,b]$-concave  stochastic orders are also useful in proving inequalities that involve convex functions. To show the usefulness of these stochastic orders in proving inequalities, we prove a novel  Hermite-Hadamard  type inequality for decreasing functions $u:[a,b] \rightarrow \R$ such that the square root of $u(x)-u(b)$ is convex (see Section \ref{sec:hh ineq}).  

There is extensive literature on stochastic orders and their applications (for a survey see \cite{muller2002comparison} and \cite{shaked2007stochastic}). The stochastic orders  we study in this paper are integral stochastic orders \citep{muller1997stochastic}. Integral stochastic orders $\succeq _{\mathfrak{F}}$ are binary relations over the set of random variables that are defined by a set of functions $\mathfrak{F}$ in the following way: for two random variables $X$ and $Y$ we have $Y\succeq _{\mathfrak{F}}X$ if and only if $\mathbb{E}[u(Y)] \ge \mathbb{E}[u(X)]$ for every $u \in \mathfrak{F}$ and the expectations exist. Many important stochastic orders are integral stochastic orders. For example, SOSD corresponds to the stochastic order $\succeq _{\mathfrak{F}}$ where $\mathfrak{F}$ is the set of all concave and increasing functions.

The integral stochastic orders we present in this paper are related to stochastic orders that weaken SOSD by restricting the set of utility functions under consideration. Third order stochastic dominance \citep{whitmore1970third} requires that the functions have a positive third derivative. Higher stochastic orders (see \cite{ekern1980increasing}, \cite{denuit1998s}, and  \cite{eeckhoudt2006putting}) restrict the sign of the functions' higher derivatives.  \cite{leshno2002preferred}, \cite{tsetlin2015generalized}, and \cite{muller2016between} restrict the values of the functions' derivatives.  \cite{vickson1977stochastic} and \cite{post2014linear} add the assumption that the  functions are in the decreasing absolute risk aversion class. \cite{post2016standard} requires additional curvature conditions on the functions' higher derivatives. 

The above stochastic orders are significantly different from the stochastic orders we introduce in this paper. All these stochastic orders impose a ranking over expectations, while the stochastic orders presented in this paper do not impose a ranking over expectations. Other known stochastic orders that do not impose a ranking over expectations are introduced in \cite{fishburn1976continua} and in  \cite{meyer1977choice}. \cite{meyer1977choice} imposes a lower and an upper bound on the Arrow-Pratt absolute risk-aversion measure (see more details on this stochastic order in Appendix \ref{section:maximal}). \cite{fishburn1976continua,fishburn1980stochastic} studies a stochastic order that is based on lower partial moments. While these stochastic orders are based on an integral condition, the main disadvantage of these stochastic orders is that their maximal generator is not known (see more details in Appendix \ref{section:maximal}).

The paper is organized as follows. In Section \ref{sec: Motivating} we study a consumption-savings problem that illustrates the usefulness of our stochastic orders. In Section \ref{sec:2} we define the $\alpha,[a,b]$-concave stochastic orders and study their properties. In Section~\ref{sec:3} we study the  applications discussed above. Section~\ref{sec:con} contains concluding remarks. The Appendix contains the proofs not presented in the main text and a discussion on the maximal generator of stochastic orders.

\subsection{A motivating application: A consumption-savings problem} \label{sec: Motivating}
Researchers have devoted a great deal of attention to analyzing the impact of future income uncertainty, in particular, on savings decisions.\footnote{For recent results see \cite{crainich2013even},  \cite{nocetti2015robust}, \cite{light2017precautionary}, \cite{lehrer2018effect},  \cite{bommier2018risk}, and \cite{baiardi2019theory}. We note that our  comparative statics results are significantly different from the results in the papers above,  because we consider the case that both the present value and the risk of future income increase. In the papers mentioned above,  stochastic orders that impose a ranking over expectations such as the second order stochastic dominance or higher order stochastic dominance are used.} In their seminal papers, \cite{leland1968saving} and \cite{sandmo1970effect} show that in a two-period consumption-savings problem for a prudent agent (i.e., an agent whose marginal utility is convex), if the labor income {\em risk} increases in the sense of second order stochastic dominance, then the agent's savings increase. The agent's savings are also affected by the {\em expected present value} of future labor income: an increase in the expected present value decreases current savings. 
Up to now, to the best of our knowledge, no stochastic order has been provided that can be used to derive comparative statics results for the case when both the present value and the risk of future income increase. For instance, consider the two labor income distributions described in Figure 2.

\begin{figure}[h]
\includegraphics[width=4cm]{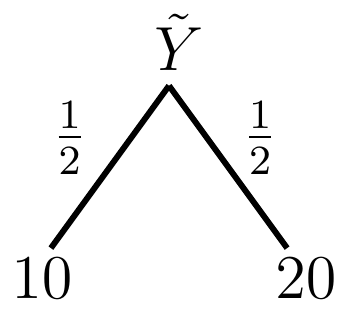}
\includegraphics[width=4cm]{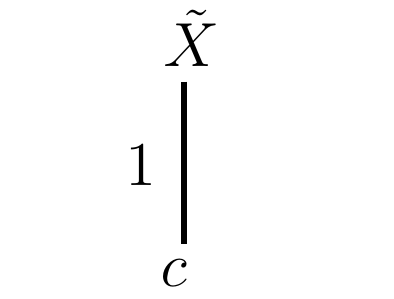}
\centering
\label{fig:2}
\caption{Future labor income}
\end{figure}

Under which distribution should we expect to observe higher savings? 
For $c\ge 15$, $\tilde{Y}$ is riskier than $\tilde{X}$, in the sense of SOSD. Thus, in the expected utility framework, savings are higher under $\tilde{Y}$ than under $\tilde{X}$ (see \cite{sandmo1970effect}). In the case that $c< 15$, it is easy to see that $\tilde{X}$ and $\tilde{Y}$ cannot be compared by SOSD. In this case, there is a {\em trade-off} between the agent's future income {\em risk} versus the agent's future income {\em present value}. Using the techniques developed in this paper we derive comparative statics results in the presence of this trade-off. We now describe the two-period consumption-savings problem that we study. 

An agent decides how much to save and how much to consume while his next period's income is uncertain. If the agent has an initial wealth of $x$ and he decides to save $0 \leq s \leq x$, then the first period's utility is given by $u(x-s)$ and the second period's utility is given by $u(Rs +y)$ where $y$ is the next period's income, $R$ is the rate of return, and $u$ describes the agent's utility from consumption. The agent
chooses a savings level to maximize his expected utility:
\begin{equation*}h (s ,F):=u (x -s) + \int  _{0}^{\overline{y}}u(Rs +y)dF(y)
\end{equation*}
where the distribution of the next period's income $y$ is given by $F$. The support of $F$ is given by $[0 ,\overline{y}]$. We assume that the agent's utility function $u$ is strictly increasing, strictly concave, and continuously differentiable.

Let $g (F) =\ensuremath{\operatorname*{argmax}}_{s \in C(x)}h(s ,F)$ be the optimal savings under the distribution $F$ where we denote by $C(x) : =[0,x]$  the interval from which the agent may choose his level of savings when his wealth is $x$.

Let $ \succeq _{I}$ be the first order stochastic dominance order and $ \succeq _{CX}$ be the convex stochastic order.\protect\footnote{
Recall that $F \succeq _{I}G$ if and only if $\int u(x)dF(x) \geq \int u(x)dG(x)$ for every increasing function $u$ and $F \succeq _{CX}G$ if and only if $\int u(x)dF(x) \geq \int u(x)dG(x)$ for every convex function $u$.
} Two well known facts about the effect of the future income's distribution on savings decisions are the following:

\begin{proposition}
(i) If $F\succeq _{I}G$ then $g(G) \geq g(F)$. 

(ii) If $F\succeq _{CX}G$ and $u^{ \prime }$ is convex then $g(F) \geq g(G)$.
\end{proposition}

Part (i) of the last Proposition states that if the future income's distribution is better in the sense of first order stochastic dominance, then the current savings are lower. The additional consumption follows from the permanent income motive, i.e., the agent wants to smooth consumption. Part (ii) of the last Proposition states that when the future  income's distribution is riskier in the sense of the convex stochastic order, then the current savings are higher. The additional savings are  called precautionary saving. 

In the Proposition \ref{prop:consumption-savings} in Section 3.1 we consider the case that the income's distribution is better (it has an higher expected value) and riskier. We show that when the agent's marginal utility is a $2,[0 ,Rx +\overline{y}]$-convex function then the precautionary saving motive is stronger than the permanent income motive. The condition that $u'$ is a $2,[0 ,Rx +\overline{y}]$-convex function guarantees that the agent's marginal utility is ``very" convex (that is, the agent is ``very" prudent) so that the agent prefers to save more under the riskier income distribution even though it has a higher expected value. This condition is satisfied by a large class of utility functions (see Section \ref{sec:3} for examples).

Our results show the potential importance of prudence as a first order consideration in policy design. If the agents are ``very" prudent, then an increase in an agent's permanent income together with an increase in future income uncertainty reduces consumption. Thus, in an economy where agents are ``very" prudent, reducing the agents' future income uncertainty can be the focus of a policy maker who aims to increase the short-run consumption. When the labor income uncertainty increases (which is a typical feature of a recession), a policy that focuses only on increasing  permanent income might lead to a decrease in consumption.

The application presented in this section uncovers two key advantages of the stochastic orders presented in this paper. First, we derive comparative statics results when an increase in the lottery's (future labor income)  riskiness increases the agent's optimal action (savings), but an increase in the lottery's expected value decreases the agent's optimal action that cannot be derived using previous results. We relate these results to the convexity of the agent's marginal utility (prudence). Second, the comparative statics results depend on the support of the distribution. The agent's marginal utility needs to be  ``very" convex only on a relevant local region of possible outcomes that depends on the agent's initial wealth level.  Hence, savings decisions' dependence on future labor income could depend on agent's wealth under our approach. Furthermore, because the $\alpha,[a,b]$-concave stochastic orders are stronger when $b$ is lower (see Section \ref{sec:2}) we can obtain sharper results for lower wealth level, i.e., we show that the precautionary motive is stronger as the wealth level decreases.


\section{The $\alpha,[a,b]$-concave stochastic order}\label{sec:2}

In this section we introduce and study the family of $\alpha,[a,b]$-concave  stochastic orders. 
We first introduce the set of $\alpha$-convex functions.\footnote{The space of $\alpha$-convex functions has been studied in the field of convex geometry (see \cite{lovasz1993random} and \cite{fradelizi2004extreme}). The $\alpha$-convex functions are also used in \cite{acemoglu2015robust} and \cite{jensen2017distributional} to derive comparative statics results in consumption-savings problems.}  
\begin{definition}\label{definition:1}
	Let $\alpha\ge 1$. We say that $u:\R\to \R_+$ is $\alpha$-convex, if $u^{\frac 1 \alpha}$ is a convex function.
\end{definition}


If $u$ is $\alpha$-convex and twice differentiable, then $u$ is $\alpha$-convex if and only if $(u(x)^{\frac 1 \alpha})'' \ge 0$. Thus, a twice differentiable function $u:\R\to \R_+$  is $\alpha$-convex if and only if $$ u(x) u''(x)\ge{ u'(x)^2 \frac{\alpha - 1}\alpha }\; \mbox{ for every }x.$$

We now introduce the set of functions that generate the family of stochastic orders that we study in this paper.\footnote{In the context of stochastic orders, one disadvantage of the set of $\alpha$-convex functions is that this set does not include the negative constant functions. This fact implies that the maximal generator of the stochastic order generated by the set of $\alpha$-convex functions might not be equal to the set of $\alpha$-convex functions. In Appendix \ref{section:maximal} we show that the stochastic order generated by the set of $\alpha$-convex functions is essentially equivalent to the second order stochastic dominance. Importantly, the set $\mathcal {I}_{\alpha ,[a,b]}$ that generates the stochastic orders we introduce in this paper is convex, closed and contain all the constant function, and hence, the set $\mathcal {I}_{\alpha ,[a,b]}$ equals its maximal generator  (see Appendix \ref{section:maximal}).} 
Let $\mathcal{B}_{[a,b]}$
be the set of bounded and measurable functions from $[a,b]$ to $\R$. For the rest of the paper we say that a function $u$ is decreasing if it is weakly decreasing, i.e., $x<y$ implies $u(x) \geq u(y)$. We say that $u$ is increasing if $-u$ is decreasing.

\begin{definition}
	Fix $\alpha\ge 1$ and $[a,b]\subseteq \R$. Let
\begin{align}
\mathcal{I}_{\alpha ,[a,b]}&=\{u\in \mathcal{B}_{[a,b]}\;| \; u \mbox{ is increasing}, \  u(b)-u(x) \mbox{ is } \alpha\mbox{-convex} \}. 
\end{align}	
 Let $F$ and $G$  be two cumulative distribution functions on $[a,b]$.\footnote{In the rest of the paper, all functions are assumed to be integrable. All the results in this paper can be extended to the case that $a = -\infty$.} We say that $F$ dominates $G$ in the $\alpha,[a,b]$-concave  stochastic order, denoted by $F \succeq_{\alph-I} G$, if for every $u\in \mathcal{I}_{\alpha ,[a,b]}$ we have
		$$\int_a^b u(x) dF(x)\ge \int_a^b u(x)dG(x).$$
\end{definition}

For the rest of the paper we say that $u$ is a $\alpha,[a,b]$-concave function if $u \in \mathcal{I}_{\alpha ,[a,b]}$ and that $u$ is a $\alpha,[a,b]$-convex function if $-u \in  \mathcal{I}_{\alpha ,[a,b]}$. 
For two random variables $X$ and $Y$ with distribution functions $F$ and $G$, respectively, we write  $X \succeq_{\alph-I} Y$ if and only if $F \succeq_{\alph-I} G$.

The family of $\alpha,[a,b]$-concave  stochastic orders generalizes second order stochastic dominance (SOSD), which corresponds to the $1,[a,b]$-concave stochastic order. For $\alpha>1$, the $\alpha,[a,b]$-concave   stochastic order is weaker than SOSD in the sense that if $X$ dominates $Y$ in the SOSD, then $X$ dominates $Y$ in the $\alpha,[a,b]$-concave  stochastic order but the converse is not true. The idea of the $\alpha,[a,b]$-concave  stochastic orders is that the inequality $\mathbb{E}[u(Y)] \ge \mathbb{E}[u(X)]$ is required to hold only for a subset of the concave and increasing functions (and not for all of them as in SOSD) in order to determine that a random variable $Y$ dominates a random variable $X$ in the $\alpha,[a,b]$-concave  stochastic order. The inequality $\mathbb{E}[u(Y)] \ge \mathbb{E}[u(X)]$ is required to hold for all the functions that belong to the set $\mathcal{I}_{\alpha ,[a,b]}$ of $\alpha,[a,b]$-concave functions which is a subset of the concave and increasing functions. As we explain below, the set $\mathcal{I}_{\alpha ,[a,b]}$ contains ``very" risk aversion decision makers where the degree of risk aversion is parameterized by $\alpha$.

\textbf{Motivation for introducing the set $\mathcal{I}_{\alpha  ,[a ,b]}$.} When $\alpha$ increases there are fewer decision makers that need to prefer $Y$ to $X$ in order to conclude that $Y$ dominates $X$ in the $\alpha  ,[a ,b]$-concave stochastic order. That is, for $\alpha _{2} >\alpha _{1}$ we have $\mathcal{I}_{\alpha _{2} ,[a ,b]} \subset \mathcal{I}_{\alpha _{1} ,[a ,b]}$ (see Proposition \ref{prop:2} in the appendix). Informally, the decision makers $u \in \mathcal{I}_{\alpha _{1} ,[a ,b]}\backslash \mathcal{I}_{\alpha _{2} ,[a ,b]}$ that are excluded from the set $\mathcal{I}_{\alpha _{1} ,[a ,b]}$ when using the $\alpha_{2},[a ,b]$-concave stochastic order instead of using the $\alpha _{1},[a,b]$-concave stochastic order are the decision makers that are the closest to being risk neutral in the set $\mathcal{I}_{\alpha _{1} ,[a ,b]}$. In other words, the decision makers $u \in \mathcal{I}_{\alpha _{1} ,[a ,b]}\backslash \mathcal{I}_{\alpha _{2} ,[a ,b]}$ have the least concave function in the set $\mathcal{I}_{\alpha _{1} ,[a ,b]}$ where the degree of concavity is measured by the elasticity of the marginal utility function with respect to the utility function. To see this, note that for a twice continuously differentiable function $u$ with the normalization $u(b)=0$,\footnote{From a decision theory point of view, we can normalize $u(b)=0$ without changing the preferences of the decision maker.}  we have $u \in \mathcal{I}_{\alpha ,[a ,b]}$ if and only if 
\begin{equation*}\frac{\partial \ln (u^{ \prime }(x))}{ \partial \ln (u(x))} =\frac{u(x)u^{ \prime  \prime }(x)}{\left (u^{ \prime }(x)\right )^{2}} \geq \frac{\alpha  -1}{\alpha } \label{Ineq: p-convex}
\end{equation*} 
for all $x \in (a ,b)$. That is, the elasticity of the marginal utility function with respect to the utility function is bounded below by $(\alpha -1)/\alpha $. The elasticity of $u'$ with respect to $u$ is a natural measure of the concavity of $u$. When the elasticity at a point $x$ is $0$, then $u$ is essentially linear around $x$. When the elasticity at a point $x$ is large, then $u$ is ``very" concave around $x$. When $\alpha$ is higher, the effect of a change in the utility function on the marginal utility function is bounded below uniformly by a higher number.

This measure of concavity has the following economic interpretation. To see this, notice that the previous inequality is equivalent to 
$$ \frac{- \frac {u''(x)}{u'(x)} } { - \frac{u'(x)}{u(x)} } = \frac{u(x)u^{ \prime  \prime }(x)}{\left (u^{ \prime }(x)\right )^{2}} \geq \frac{\alpha  -1}{\alpha }$$
for all $x\in (a,b)$. Thus, for an agent with an $\alpha,[a,b]$-concave utility function the sensitivity to risk, measured by the coefficient of risk aversion, is at least $(1-1/\alpha)$ times the sensitivity to reward, measured by the marginal utility divided by the level of utility. In other words, the parameter $(1-1/\alpha)$ bounds how much the agent 
prefers an increase in reward when it is accompanied by an increase in risk. 

Furthermore, in many of our applications (e.g., the consumption-savings problem) we impose that the agents'  marginal utility function is positive and $\alpha,[a,b]$-convex (i.e., $-u'\in\mathcal{I}_{\alpha,[a,b]}$) to derive our comparative statics results.  For such agents, their utility functions $u$ satisfy 
$$\frac{\;- \frac {u'''(x)}{u''(x)}\; }{\; -\frac{u''(x)}{u'(x)} \; }\ge  \frac{(u'(x)-u'(b))u'''(x)}{(u''(x))^2} \ge \frac {\alpha-1}{\alpha}$$
for all $x\in (a,b)$.\footnote{The first inequality holds since $u'$ is positive and convex. The second inequality comes from the characterization of $\alpha$ convexity for smooth functions.} Thus, for this class of agents we have that their coefficient of prudence is at least $(1-1/\alpha)$ times their coefficient of risk aversion. We show in the applications that this condition implies sharp comparative statics results. Also, using this inequality, we can bound the degree of convexity $\alpha$ by using estimates of the coefficients of prudence and risk aversion.

\textbf{ Examples.} The following examples show that the family of $\alpha,[a,b]$-concave  stochastic orders allows us to compare simple lotteries that are not comparable by other popular stochastic orders. 

In Example \ref{Example 1} we show that $Y \succeq _{\alpha ,[a ,b] -I}X$ for the random variables in Figure \ref{fig:1} (see Section 1). This example is simple and can be used in order to design a simple experiment to determine if the decision maker's utility function is not $\alph$-concave function for some $\alpha$. 

We provide two more examples of random variables $X$ and $Y$ where $X$ has a higher expected value and is riskier than $Y$, and $Y$ dominates $X$ in the $\alpha  ,[a ,b]$-concave stochastic order. The second example involves compound lotteries and the third example involves a uniform distribution.

\begin{example} \label{Example 1}
Consider two lotteries $X$ and $Y$. Lottery $X$ yields $a$ dollars with probability $\lambda^{\alpha}$ and $b$ dollars with probability $1 -\lambda^{\alpha }$ where $b>a$ and $\alpha  \geq 1$. Lottery $Y$ yields $\lambda a +\left (1 -\lambda \right )b$ dollars with probability $1$. Then  $Y \succeq _{\alpha ,[a ,b] -I}X$.\footnote{The proofs of this assertion and the assertions in Example 2 and 3 are presented in the appendix.}   
 
\end{example}

\begin{example} \label{Example 2}
(Compound lotteries). Consider two lotteries $Y$ and $X$. Lottery $Y$ yields $x_{i} : =\lambda _{i}a +\left (1 -\lambda _{i}\right )b$ with probability $0 <p_{i} <1$, $i =1 ,\ldots  ,n$ where $0 <\lambda _{1} <\ldots  <\lambda _{n} <1$. Lottery $X$ yields $a$ with probability $\sum _{i}p_{i}\lambda _{i}^{\alpha }$ and $b$ with probability $1 -\sum _{i}p_{i}\lambda _{i}^{\alpha }$. Then $Y \succeq _{\alpha ,[a ,b] -I}X$.  

\end{example}

\begin{example} (Uniform distribution).\label{Example 3}
Consider two lotteries $Y$ and $X$. Lottery $X$ yields $a$ dollars with probability $\frac{1}{\alpha  +1}$ and $b$ dollars with probability $\frac{\alpha }{\alpha  +1}$ where $b >a$ and $\alpha  \geq 1$. Lottery $Y$ is uniformly distributed on $[a ,b]$. Then $Y \succeq _{\alpha  ,[a ,b] -I}X$. 
\end{example}

For general distribution functions $F$ and $G$ it is not trivial to check whether $F$ dominates $G$ in the $\alpha,[a,b]$-concave  stochastic order. Below we provide a sufficient condition which is given by a simple integral inequality which guarantees that $F$ dominates $G$ in the $\alpha,[a,b]$-concave stochastic order.

\textbf{Properties of the $\alpha,[a,b]$-concave stochastic orders.} In Proposition \ref{prop:2.5} we  provide some properties of the $\alpha,[a,b]$-concave  stochastic order. The first property is intuitive and  shows that  $F \succeq_{\alph-I}G$ implies $F\succeq_{\beta,[a,b]-I}G$ whenever $\beta>\alpha$. This is immediate because for $\alpha _{2} >\alpha _{1}$ we have $\mathcal{I}_{\alpha _{2} ,[a ,b]} \subset \mathcal{I}_{\alpha _{1} ,[a ,b]}$. Importantly, $F\succeq_{1,[a,b]-I}G$, implies $F\succeq_{\alph-I} G$ for every $\alpha \geq 1$.  That is, the  $\alpha ,[a,b]$-concave stochastic order is weaker than the second order stochastic dominance for every $\alpha > 1$. The second property relates to translation invariance. We note that if $X\succeq_{\alph-I}Y$ then $X+c\succeq_{\alpha,[a,b]-I}Y+c$ is not necessarily well defined for $c \in \mathbb{R}$. This is because the $\alpha ,[a,b]$-concave stochastic orders are defined on a specific support. Hence,  translation invariance for the $\alph$-concave stochastic orders means that $X\succeq_{\alph-I}Y$ implies $X+c\succeq_{\alpha,[a+c,b+c]-I}Y+c$ for every $c \in \mathbb{R}$. This is exactly the second property. Translation invariance is an important property for a stochastic order  because it means that comparisons between random variables that are based on this stochastic order are invariant for currency conversions and for the addition of a sure amount of wealth (many important stochastic orders satisfy the translation invariance property, e.g., SOSD). 
The third property shows that $F \succeq_{\alpha,[a,b']-I} G $ implies $ F\succeq_{\alph-I} G$ whenever  $b'\ge b$. This is important in applications because we can choose a support such that all the random variables of interest are defined on this support. For example, we use Property 3 to prove our results in the consumption-savings problem discussed in the introduction (see Proposition \ref{prop:consumption-savings}).  

\begin{proposition}\label{prop:2.5} 
	The following properties hold:

 1. Let $\beta>\alpha$. Then $F \succeq_{\alph-I}G$ implies $F\succeq_{\beta,[a,b]-I}G$. 
		 
2.  Suppose that $X\succeq_{\alph-I}Y$. Then $X+c\succeq_{\alpha,[a+c,b+c]-I}Y+c$  for every $c\in \R$.
		 
3. Suppose that $F$ and $G$ are distributions on $[a,b]$. Then for every $b'\ge b$ we have 
\begin{align*}
		F \succeq_{\alpha,[a,b']-I} G &\Longrightarrow  F\succeq_{\alph-I} G. 
		 \end{align*}

	\end{proposition}


\textbf{ A sufficient condition for domination in the $\alpha,[a,b]$-concave stochastic order.}  \label{sec:2.3} 
Even though the set of $\alpha,[a,b]$-concave functions has a clear economic motivation as we explained above, the geometry of this set is complicated. Therefore,  a simple characterization is unlikely to exist. For this reason we now introduce a simple integral condition to check whether a distribution $F$ dominates a distribution $G$ in the $\alpha,[a,b]$-concave stochastic order.\footnote{We note that we cannot use similar numerical methods to the ones developed in \cite{post2013general} and \cite{fang2017higher} to characterize the $\alpha,[a,b]$-concave stochastic order. The reason is that the methods in \cite{post2013general} and \cite{fang2017higher} are developed for stochastic orders  generated by functions that are defined by inequalities  that are linear with respect to the functions' derivatives. In contrast, the $\alpha,[a,b]$-concave functions cannot be defined by inequalities  that are linear with respect to derivatives.} 
This integral condition generates a new stochastic order $\succeq_{n,[a,b]-S}$ for $n\in \N$, which we call the $n,[a,b]$-sufficient stochastic order and is of independent interest.

\begin{definition}
	Consider two distributions $F$ and $G$ over $[a,b]$ and a positive integer $n$. We say that $F$ dominates $G$ in the $n,[a,b]$-sufficient stochastic order, and write $F \succeq_{n,[a,b]-S} G$ for $n\in \N$, if and only if for all $c =(c_{1} , \ldots ,c_{n})\in [a,b]^n$ we have $$
	\int_a^b \prod_{i=1}^n \max\{c_i-x,0\} dF(x) \leq \int_a^b \prod_{i=1}^n \max\{c_i-x,0\} dG(x).$$
\end{definition}

Note that $F$ dominates $G$ in the $1,[a,b]$-sufficient stochastic order if and only if $F$ dominates $G$ in the second order stochastic dominance, i.e., $\int_{a}^{b}u(x)dF(x) \geq  \int_{a}^{b}u(x)dG(x)$ for every concave and increasing function $u$ (see Theorem 1.5.7. in \cite{muller2002comparison}). In Proposition \ref{prop:3} we extend this result. We show that if $F$ dominates $G$ in the $n,[a,b]$-sufficient stochastic order, then $F$ dominates $G$ in the $n,[a,b]$-concave stochastic order for all $n \in \mathbb{N}$. Combining this with Proposition \ref{prop:2.5} part (i) we conclude that  if $F$ dominates $G$ in the $n,[a,b]$-sufficient stochastic order, then $F$ dominates $G$ in the $\alpha,[a,b]$-concave stochastic order for all $1 \leq \alpha \leq n$. 

Thus, Proposition \ref{prop:3} provides a simple integral condition that guarantees domination in the $\alph$-concave stochastic orders.

\begin{proposition}\label{prop:3}
Consider two distributions $F$ and $G$ over $[a,b]$ and $n= \lceil \alpha \rceil$. Then $F \succeq _{n,[a,b]-S}G$ implies $F \succeq_{\alpha,[a,b]-I} G$.
\end{proposition}

For $n>1$ the converse of Proposition \ref{prop:3} does not hold. That is, $F \succeq_{n,[a,b]-I} G $ does not imply $F\succeq_{n,[a,b]-S} G$. For example, for $n=2$ it can be checked that the function $- \max\{c_{1}-x,0\}\max\{c_{2}-x,0\}$ is not a $2,[a,b]$-concave  function for $c_{2} \neq c_{1}$. Because the maximal generator of the $2,[a,b]$-concave  stochastic order is the set of $2,[a,b]$-concave functions (see Section \ref{section:maximal} in the appendix) we conclude that $F\succeq_{2,[a,b]-I} G$ does not imply $F \succeq_{2,[a,b]-S} G$. On the other hand, in Section \ref{section:maximal} in the appendix we formally prove a partial characterization of the $2,[a,b]$-sufficient stochastic. We show that this stochastic order generates an appealing set of functions (in particular, it is not equivalent to SOSD as our examples show).

The $n,[a,b]$-sufficient stochastic order is particularly useful for the case $n=2$. We now provide a sufficient condition that ensures that $F$ dominates $G$ in the $2,[a,b]$-concave  stochastic order by applying Proposition~\ref{prop:3}. Similar conditions are used to determine if $F$ dominates $G$ in other popular stochastic orders such as the second order stochastic dominance and the third order stochastic dominance.

\begin{proposition}\label{prop:4}
	Consider two distributions $F$ and $G$ over $[a,b]$.
	We have that $F\succeq_{2,[a,b]-S} G$ if and only if for all $c\in [a,b]$ the following two inequalities hold:
	\begin{align}
	&	(b-c)\bigg[\int_a^{c} F(x) dx -\int_a^{c} G(x) dx  \bigg] + 2\int_a^{c} \bigg(\int_a^x F(z)dz-\int_a^x G(z)dz\bigg)dx  \leq 0 \label{ine:1} \\
	&	\int_a^{c} \bigg(\int_a^x F(z)dz-\int_a^x G(z)dz\bigg)dx  \leq 0 \label{ine:2}\;.
	\end{align}
\end{proposition}

Interestingly, the conditions of Proposition~\ref{prop:4} inherently relate to SOSD and third order stochastic dominance. Third order stochastic dominance corresponds to inequality (\ref{ine:2}) and to the condition $\E_G[X]\leq  \E_F[X]$. SOSD corresponds to $\int_a^{c} F(x) dx -\int_a^{c} G(x) dx \leq 0 $ for all $c$ (which also implies a ranking over expectations and inequality (\ref{ine:2})). In contrast, the $2,[a,b]$-sufficient stochastic order does not imply a ranking over expectations, but the left-hand-side of inequality (\ref{ine:1}) is bounded from above by a number smaller than $0$ for any value of $c$  that implies a violation of SOSD, i.e., any $c$ that satisfies $\int_a^{c} F(x) dx -\int_a^{c} G(x) dx > 0 $. Thus, the $2$-sufficient stochastic order does not imply the second or the third  stochastic order.

In some cases, the $2$-sufficient stochastic order provides a necessary and sufficient integral condition to conclude that $F\succeq_{2,[a,b]-I} G$. If  condition~\eqref{ine:2} implies  condition~\eqref{ine:1} then condition~\eqref{ine:2} holds if and only if $F\succeq_{2,[a,b]-I} G$. To see this, note that for $c \in [a,b]$ the function $- \max\{c-x,0\}^{2}$ is a $2,[a,b]$-concave function and that 
\begin{equation*}
    \int_a^b \max\{c-x,0\}^{2} dF(x)=2\int_a^{c} \bigg(\int_a^x F(z)dz\bigg)dx 
\end{equation*}
(see Lemma \ref{lemma:0} in the appendix). Thus, if $F\succeq_{2,[a,b]-I} G$ holds, then condition~\eqref{ine:2} holds. On the other hand, if condition~\eqref{ine:2} implies condition~\eqref{ine:1}, then from Proposition \ref{prop:4} we have $F\succeq_{2,[a,b]-I} G$. We summarize this result in the following Corollary. 

\begin{corollary}\label{coro:1}
Let $F$ and $G$ be two distributions over $[a,b]$. Suppose that if condition~\eqref{ine:2} holds then condition~\eqref{ine:1} also holds.  Then condition~\eqref{ine:2} holds if and only if $F\succeq_{2,[a,b]-I} G$.
\end{corollary}


Corollary \ref{coro:1} provides a tool to show that a random variable dominates another random variable in the $2,[a,b]$-concave stochastic order. We will provide  applications of Corollary \ref{coro:1} in Section \ref{sec:3}).

\section{Applications}\label{sec:3}

In this section, we discuss four applications in which we use the  $\alpha,[a,b]$-concave  stochastic orders: a consumption-savings problem with an uncertain future income, self-protection problems, a Diamond-type search model with one-sided incomplete information, and comparing uniform distributions. 

\subsection{Precautionary saving when the future labor income is riskier and has a higher expected value} \label{sec:Precu}

Consider the consumption-savings problem described in Section \ref{sec: Motivating}. Recall that
 $$g (F) =\ensuremath{\operatorname*{argmax}}_{s \in C(x)}h(s ,F)$$ is the optimal savings under the distribution $F$ where we denote by $C(x) : =[0,x]$  the interval from which the agent may choose his level of savings when his wealth is $x$ (see Section \ref{sec: Motivating}).

In the following Proposition we show that when the agent's marginal utility is a $2,[0 ,Rx +\overline{y}]$-convex function, then the precautionary saving motive is stronger than the permanent income motive, i.e., $F \succeq _{2 ,[0 ,Rx +\overline{y}] -I}G$ implies $g(G) \geq g(F)$. That is, when $F$ is better and riskier than $G$ in terms of the $2,[0,Rx +\overline{y}]$-concave stochastic order, then  savings under $G$ are higher than under $F$. 

Proposition \ref{prop:consumption-savings} uncovers the potential importance of prudence and future income uncertainty as first order considerations in policy design. If the agents are ``very" prudent, then an increase in an agent's permanent income together with an increase in future income uncertainty reduces consumption. Hence, in an economy where agents have ``very" convex marginal utilities, i.e., agents are ``very" prudent, reducing the agents' future income uncertainty can be the major focus of a policy maker who aims to increase the short-run consumption. A policy that increases permanent income can lead to a decrease in the short-run  consumption when the future labor income uncertainty increases (which is a typical feature of a recession).

The condition that $u'$ is a $2,[0 ,Rx +\overline{y}]$-convex function is not satisfied by the important class of constant relative risk aversion functions. However, a closely related class of utility functions satisfies this condition. It can be shown that $u'$ is a $2,[a,b]$-convex  function for the utility function $$u(x)=\frac{x^{1-\gamma}}{1-\gamma} + \frac{\gamma x^{2}}{2b^{\gamma+1}}$$
for $\gamma > 0$,  $\gamma \neq 1$ and $u(x)= \log(x) +x^{2}/2b^{2}$ for $\gamma = 1$. Note that for a large $b$ the utility function defined above is close to a constant relative risk aversion utility function.

\begin{proposition}\label{prop:consumption-savings}
 Suppose that $u^{ \prime }$ is a $2 ,[0 ,Rx +\overline{y}]$-convex function.
 
(i) If $F  \succeq _{2 ,[0 ,Rx +\overline{y}] -S}G$ then $g(G) \geq g(F)$, i.e., the savings under $G$ are greater than or equal to the savings under $F$.   

(ii) If $F  \succeq _{2 ,[0 ,Rx +\overline{y}] -I}G$ then $g(G) \geq g(F)$, i.e., the savings under $G$ are greater than or equal to the savings under $F$.  
\end{proposition}

For a thrice differentiable utility function,  the condition that $u'$ is a $2,[0,b]$-convex decreasing function is equivalent to the condition that $(u'(x)-u'(b))u'''(x)/(u''(x))^{2} \geq 0.5$ for all $x \in [a,b]$.  The last expression means that the ratio between the coefficient of relative prudence and the coefficient of relative risk aversion is bounded below by $1/2$ (see the discussion in Section \ref{sec:2}). In this case, the precautionary effect is stronger than the permanent income effect. 

We stated the result in Proposition  \ref{prop:consumption-savings} also with respect to the sufficient stochastic order (see part (i)). This can be useful in applications because it easy to check whether $F$ dominates $G$ in the sufficient stochastic order.

\subsection{Self-protection problems} \label{sec:self-prot}
Self-protection is a costly action that reduces the probability for a loss (see \cite{ehrlich1972market}). Since the work of \cite{ehrlich1972market}, self-protection problems are widely studied in the literature on decision making under uncertainty.\footnote{For example, see \cite{dionne1985self}, \cite{eeckhoudt2005impact}, \cite{meyer2011diamond}, \cite{denuit2016tradeoffs}, and \cite{liu2017increasing}.} Should a decision maker choose more or less self-protection? One way to answer this question is based on stochastic orders. A risk-averse decision maker can decide to prefer less self-protection if most risk-averse decision makers prefer less self-protection. In this section we study a standard self-protection problem and provide a decision rule to answer the question above based on the $2,[a,b]$-concave stochastic order. We find conditions that imply that an agent prefers to decrease the level of self-protection even when the increase in self-protection is profitable in expectation.   

We study a simple self-protection problem (as in \cite{ehrlich1972market} and \cite{eeckhoudt2005impact}) where there are two possible outcomes: a loss of fixed size or no loss at all. We now provide the formal details. 

There are two lotteries $X$ and $Y$. Lottery $X$ yields $w-L-e_{x}$ with probability $p$ and $w-e_{x}$ with probability $1-p$. Lottery $Y$ yields $w-L-e_{y}$ with probability $q$ and $w-e_{y}$ with probability $1-q$. The wealth that the decision maker has is given by  $w$, the fixed loss is given by $L$, and $p$ and $q$ are the probabilities of loss that depend on the level of expenditure on self-protection $e_{i}$ for $i=x,y$. We assume that $e_{x}>e_{y}$ and $q>p$. That is, if the decision maker chooses a higher expenditure on self-protection, then the probability of a loss decreases. We also assume that $w-e_{x} > w-L-e_{y}$. If the last inequality does not hold, every rational decision maker would clearly prefer $Y$ to $X$. The following Proposition follows immediately from Lemma \ref{Lemma:two-point} and part (i) of Proposition \ref{prop:2.5}.

\begin{proposition} \label{Prop:Self-protection}
    Suppose that the expected value of $X$ is higher than the expected value of $Y$, i.e., $-pL-e_{x} \geq -qL -e_{y}$. 
    Then
    \begin{equation}\label{Ineq:self-protection}
        p(e_{x}-e_{y}+L)^{2}+(1-p)(e_{x}-e_{y})^{2} \geq qL^{2} 
        \end{equation}
        if and only if $Y\succeq_{2,[w-L-e_{x},w-e_{y}]-I} X$, i.e., $Y$ dominates $X$ in the $2,[w-L-e_{x},w-e_{y}]$-concave  stochastic order.  
    
\end{proposition}

The interpretation of inequality (\ref{Ineq:self-protection}) is straightforward. For simplicity, normalize $e_{y}$ to be $0$ so $e_{x}$ is the amount that the agent can spend on self-protection to decrease the probability of a loss to $p$. In this case, the agent has a wealth of $w$ in any realization. If the agent does not spend on self-protection, then the random variable $\tilde{Y}$ that yields $L$ with probability $q$ and $0$ with probability $1-q$ represents the agent's future loss. When the agent chooses to spend $e_{x}$ on self-protection then the random variable $\tilde{X}$ that yields $e_{x}+L$ with probability $p$ and $e_{x}$ with probability $1-p$ represents the future loss (in this case the agent loses the expenditure on self-protection $e_{x}$ in any outcome). Our results show that if the expected loss under $\tilde{Y}$ is higher than under $\tilde{X}$ and the second moment of $\tilde{X}$ is higher than the second moment of $\tilde{Y}$ then the decision maker should not spend on self-protection according to the decision rule that is based on the $2,[w-L-e_{x},w]$-concave stochastic order. That is, if spending on self-protection increases the risk (captured by the second moment) of future loss, then the decision maker does not increase the expenditure on self-protection even when the increase in  self-protection increases the expected value of the decision maker's final wealth. 

In the self-protection problem that we study in this section, a simple condition that relates to the distributions'  first and second moments  captures the trade off between expected value and riskiness that the $2$-concave stochastic order provides. We show in the appendix that this is true for general distributions whose supports contain exactly two elements.\footnote{In the special case of distributions whose supports contain exactly two elements, conditions on the first two moments imply domination in the $2,[a,b]$-concave stochastic order (see Lemma \ref{Lemma:two-point} in the appendix). This is somewhat intuitive because information on the first two moments essentially determines the distributions. 
}

\subsection{A Diamond-type search model with one-sided incomplete information}
In this section we a study a Diamond-type search model studied in \cite{diamond1982aggregate}
and \cite{milgrom1990rationalizability} to a one-sided incomplete information
framework. Consider the case of two agents that benefit from a match. We analyze the case where one player has better information than the other. For instance,
one player has been in the market for a long time and his type is known, whereas the second player just entered the market, so his type is unknown. We show
that a shift (in the sense of the $\alpha,[a ,b]$-concave stochastic order) in the uninformed player's beliefs about the informed player's type leads to an increase in the highest
equilibrium probability of matching. We now describe the Bayesian game. 

There are two players who exert efforts in order to find a
match. Each player exerts a costly effort $e_{i} \in E:=[0 ,1]$, in order to achieve a match. For each player, the value of a match is one. The probability of matching is $e_{1} e_{2}$, given the efforts $e_{1} ,e_{2}$.  The cost of Player $1$'s (the uninformed player) effort is given by a strictly convex and strictly increasing function $c_{1} (e)$ that is known to both players. The cost of Player $2$'s effort is given by $c_{2} (e ,\theta ):=\frac{e_{2}^{k +1}}{(k +1)(1 -\theta )^{l}}$ where $\theta  \in [0 ,1)$ is Player $2$'s type which is not known to Player $1$, and $k ,l >0$ are some parameters. Player $1$'s beliefs about the value of $\theta $ are given by a distribution $F$ with support on $[0 ,1)$. 

Standard arguments show that this game is a supermodular game, and thus the highest and the lowest equilibria
exist (see \cite{topkis1979equilibrium} and the Appendix for more details).\footnote{The solution concept we use is the standard Bayesian Nash equilibrium. We define it formally in the appendix. }
Define $\bar{m} (F) =\bar{e}_{1}^{ \ast } \bar{e}_{2}^{ \ast } (\theta )$ to be the highest equilibrium probability of matching. Under certain parameters, we show that a shift in Player 1's beliefs, in the sense of the $\alpha  ,[0 ,1]$-concave stochastic order, leads to an increase in the highest equilibrium probability of matching.

\begin{proposition}
\label{prop:7} \label{Prop: game}Fix $\alpha  \geq 1$. Suppose that $l \geq \alpha k$. If $F^{ \prime }  \succeq _{\alpha ,[0 ,1] -I}F$ then $\bar{m} (F^{ \prime }) \leq \bar{m} (F)$. That is, the highest equilibrium probability of matching is decreasing with respect to the $\alpha,[0,1]$-concave stochastic order. 
\end{proposition}

We note that Proposition \ref{Prop: game}
allows us to derive non-trivial comparative statics results. Assume that $F^{ \prime }  \succeq _{\alpha  ,[0 ,1] -I}F$. $F^{ \prime }$ might have a lower expected value than $F$, which means that the uninformed player thinks that the informed player's cost has a lower expected value. Thus, the uninformed
player should increase his effort, since he is expecting that the informed player will increase his effort. On the other hand, $F^{ \prime }$ is less riskier than $F$ and this induces the uninformed player to decrease his effort. Proposition \ref{Prop: game}
shows that ultimately, in equilibrium, the latter effect is stronger. Thus, the efforts of both players decrease and the equilibrium probability of matching
decreases. 

\subsection{Uniform distributions and inequalities for $2,[a,b]$-convex decreasing  functions} \label{sec:hh ineq}
Convex functions are fundamental in proving many well-known inequalities. The convex stochastic order is a powerful tool for proving inequalities that involve convex functions (see \cite{rajba2017some} for a survey). In this section we prove  inequalities for convex functions that belong to the set $\mathcal{I}_{2 ,[a ,b]}$ using the $2 ,[a ,b]$-concave stochastic order. We first compare two general uniform distributions that are of independent interest.

We consider two uniform random variables. Suppose that $G \sim U[a_1,b_1]$ and $F \sim U[a_2,b_2]$ where $U[a,b]$ is the continuous uniform random variable on $[a,b]$. 
The following Lemma provides a necessary and sufficient condition on the parameters $(a_1,b_1,a_2,b_2)$ so that $F\succeq_{2,[a_1,b_1]-I} G$.
\begin{lemma}\label{Lemma:1}
	 Suppose that $G\sim U[a_1,b_1]$ and $F\sim U[a_2,b_2]$. Assume that $a_1<a_2<b_2<b_1$ and  $\frac{a_1+b_1}2>\frac{a_2+b_2}2$.\footnote{If this does not hold then the expected value of $F$ is higher or equal to the expected value of 
$G$, so we clearly have $F\succeq_{1,[a_1,b_1]-I} G$. That is, $F$ dominates $G$ in the second order stochastic dominance. } Then $F\succeq_{2,[a_1,b_1]-I} G$ if and only if 
	\begin{equation}\label{ine:dist}
	b_1 \le \frac{3(a_2+b_2)-2a_1 + \sqrt{a_2^2+10a_2b_2+b_2^2-12a_1(a_2+b_2-a_1)}}4 \;.
	\end{equation}
\end{lemma}

Lemma \ref{Lemma:1} can be used to prove non-trivial inequalities that involve concave functions. The lemma implies that if inequality (\ref{ine:dist}) holds, then for every $2,[a_{1},b_{1}]$-concave function $u$ we have $$\int_{a_{2}}^{b_{2}} u(x) dF(x)\ge \int_{a_{1}}^{b_{1}} u(x)dG(x).$$

We leverage Lemma \ref{Lemma:1} to prove Hermite-Hadamard inequalities for $2 ,[a ,b]$-concave functions. Hermite-Hadamard inequalities   are important in the literature on inequalities and have numerous applications in various fields of mathematics (see \cite{peajcariaac1992convex} and \cite{dragomir2003selected}).
Recall that the classical Hermite-Hadamard inequality states that for a convex function $f :\left [a ,b\right ] \rightarrow \mathbb{R}$ we have\begin{equation}f\genfrac{(}{)}{}{}{a +b}{2} \leq \frac{1}{b -a}\int _{a}^{b}f(x)dx \leq \frac{f\left (a\right ) +f(b)}{2} . \label{Ineq: HH}
\end{equation}

Inequality (\ref{Ineq: HH}) is an easy consequence of the convex stochastic order. The left-hand-side of the inequality states that the uniform random variable on $[a ,b]$ dominates the random variable that yields an amount of $(a +b)/2$ with probability $1$ in the sense of the convex stochastic order. The right-hand-side of the inequality states that the uniform random variable on $[a ,b]$ is dominated by the random variable that yields $a$ and $b$ with probability $1/2$ each in the sense of the convex stochastic order. Using a similar stochastic orders approach, we now extend and improve this inequality for functions $f \in - \mathcal{I}_{2 ,[a ,b]}$.

\begin{proposition}\label{prop:HH ineq}
Suppose that $f \in - \mathcal{I}_{2 ,[a ,b]}$ where $a <b$. Then 
\begin{equation}f\left (\gamma b +\left (1 -\gamma \right )a\right ) \leq \frac{1}{b -a}\int _{a}^{b}f\left (x\right )dx \leq tf\left (a\right ) +\left (1 -t\right )f(b)
\end{equation}for all $t \geq \frac{1}{3}$ and for all $\gamma  \geq \frac{2}{3 +\sqrt{3}}$. 
\end{proposition}

\section{Concluding Remarks}\label{sec:con}

In this paper, we introduce the $\alpha,[a,b]$-concave  stochastic orders, a new family of stochastic orders that generalizes the second order stochastic dominance.  The $\alpha,[a,b]$-concave  stochastic orders provide a tool for deriving comparative statics results in applications from the economics literature that cannot be obtained using previous stochastic orders. We illustrate this in three different applications: Consumption-savings problems, self-protection problems and Bayesian games. 
We provide a simple sufficient conditions to ensure domination in the $\alpha,[a,b]$-concave stochastic order when $\alpha$ is a positive integer. We foresee additional beneficial applications of $\alpha,[a,b]$-concave stochastic orders, especially for comparing lotteries that have different expected values and different levels of risk.

\appendix

\section{Appendix: The maximal generator and other stochastic orders}\label{section:maximal}

 In this section we discuss the maximal generator of an integral stochastic order and discuss other stochastic orders that do not impose a ranking over the expectations of the random variables in consideration. We now define the maximal generator of an integral stochastic order.

Define $F \succeq _{\mathfrak{F}}G$ if \begin{equation*}\int _{a}^{b}u(x)dF(x) \geq \int _{a}^{b}u(x)dG(x)
\end{equation*}
for all $u \in \mathfrak{F}$ where $\mathfrak{F}\subseteq\mathcal{B}_{[a,b]}$.  The stochastic order $\succeq _{\mathfrak{F}}$ is called an integral stochastic order.

The maximal generator $R_{\mathfrak{F}}$ of the integral stochastic order $ \succeq _{\mathfrak{F}}$ is the set of all functions $u$ with the property that $F \succeq _{\mathfrak{F}}G$ implies \begin{equation*}\int _{a}^{b}u(x)dF(x) \geq \int _{a}^{b}u(x)dG(x).
\end{equation*}

\cite{muller1997stochastic} studies the properties of the maximal generator. In our context, Muller's results imply that the following Proposition holds. 

\begin{proposition}\label{Prop: maximal generator}
    (Corollary 3.8 in \cite{muller1997stochastic}). Suppose that $\mathfrak{F}\subseteq\mathcal{B}_{[a,b]}$ is a convex cone containing the constant functions and is closed under pointwise convergence. Then  $R_{\mathfrak{F}}=\mathfrak{F}$.
\end{proposition}

 From a decision theory point view, when using a stochastic order to determine whether a random variable is better or riskier than another random variable, it is important to characterize the maximal generator. If the maximal generator is not known, it is not clear what utility functions are under consideration when deciding if a random variable is better or riskier than another random variable. 

From Proposition~\ref{prop:2}, we have that $\mathcal{I}_{\alph}$ is a convex cone that is closed in the topology of pointwise convergence. Also, the set $\mathcal{I}_{\alph}$ contains all the constant functions. Hence, from Proposition \ref{Prop: maximal generator} we conclude that the maximal generator of the $\alpha,[a,b]$-concave  stochastic order $\succeq_{ \alph-I }$ is the set $\mathcal{I}_{\alph} $. 




We now show that a stochastic order that is based on the $\alpha$-convex and decreasing functions do not lead to an interesting new stochastic order. The reason is that the maximal generator of this stochastic order includes all the convex, positive, differentiable and decreasing functions (see Proposition \ref{prop:alpha-maximal} below). Hence, this stochastic order is essentially equivalent to SOSD. This result shows that studying stochastic orders that their maximal generator is unknown could be misleading. 

\begin{definition}
Consider two distributions $F$ and $G$ on $[a,b]$. We say that $F$ dominates $G$ in the $\alpha$-convex stochastic order, denoted by $F\succeq_{\alpha-DCX}G$, if for every decreasing and $\alpha$-convex function $u:[a,b] \to \R_+$, we have
$$\int_a^b u(x) dF(x) \ge \int_a^b u(x)  dG(x)\;.$$
\end{definition} 
Notice that the functions under consideration in this order have a constraint over the range: every function $u$ has to be non-negative. 

\begin{proposition}\label{prop:alpha-maximal}
	Let $\alpha > 1$. Then
	$F \succeq_{\alpha-DCX} G$ implies that for every convex and decreasing function $u:[a,b] \rightarrow \R_{+}$ that is twice differentiable we have 
	$$\int_a^b u(x) dF(x) \ge \int_a^b u(x)  dG(x).$$
\end{proposition} 
 
The above proposition shows that the $\alpha$-convex stochastic order is essentially the same as the well studied convex and decreasing stochastic order. Note that the set of decreasing $\alpha$-convex functions is a closed convex cone that is a strict subset of the set of decreasing convex functions. This, nevertheless, is not a contradiction of Proposition \ref{Prop: maximal generator}, because negative constant functions do not belong to the set of $\alpha$-convex functions. This fact also explains the proof of Proposition \ref{prop:alpha-maximal}. Informally, for every convex function $u$, there exists a constant $c > 0$ such that $u+c$ is essentially $\alpha$-convex. 

The above discussion is the reason that we introduce the set of functions $\mathcal{I}_{\alph}$ (and the related set $- \mathcal{I}_{\alph} $) which include all the constant functions. One limitation of these sets is that if $u \in \mathcal{I}_{\alph}$ and $u$ is twice differentiable, then $u'(b)=0$ (see Proposition \ref{prop:2}). That is, the decision makers under consideration when comparing two random variables have a $0$ marginal utility at the point $b$. One way to overcome this is to choose a large $b'$ such that it is plausible to assume that $u'(b^{\prime})=0$. Then, if $F$ and $G$ are distributions on $[a,b]$, we can use the fact that $F \succeq_{\alpha,[a,b']-I} G \Longrightarrow  F\succeq_{\alph-I} G$ (see Proposition \ref{prop:2.5}) to conclude that $F\succeq_{\alph-I} G$. 

\subsection{The $2$-sufficient stochastic order}

In this section we provide a partial characterization of the $2$-sufficient stochastic order. 
Recall that $F$ dominates $G$ in the $2,[a,b]$-sufficient stochastic order, i.e., $F \succeq_{2,[a,b]-S} G$ if and only if for all $c =(c_{1} ,c_{2})\in [a,b] \times [a,b]$ we have $$
	\int_a^b \max\{c_1-x,0\}\max\{c_2-x,0\} dF(x) \leq \int_a^b \max\{c_1-x,0\}\max\{c_2-x,0\} dG(x).$$
Hence, the $2,[a,b]$-sufficient stochastic order is generated by a simple integral inequality that naturally generalizes SOSD and is of independent interest. It is interesting to know that the maximal generator of this stochastic order. The following Proposition is a first step in this direction. We show that the sum of $2,[a,b]$-concave functions and functions with a bounded below Arrow-Pratt measure of risk aversion essentially contains the maximal generator of the  $2,[a,b]$-sufficient stochastic order.

Define the set of functions 
\begin{equation*}{AP}_{2,[a,b]} : =\{u \in C^{2}([a,b]): u'(x) \geq 0, \:  u''(x) \leq 0, \:  u'(x) + u''(x)(b -x) \leq 0  \text{ } \; \forall x \in (a,b)  \}.
\end{equation*}
Note that if $u'>0$  and $u \in AP_{2,[a,b]}$, then $- u'' / u' \geq 1/(b-x)$, i.e., the Arrow-Pratt measure of risk aversion of $u$ is bounded below by $1/(b-x)$. For two sets $U$ and $U'$ define $U+U' := \{ u + u': \ u \in U \ u' \in U'\}$. 

\begin{proposition} \label{prop: maximal of sufficient}
  Suppose that $\int_a^b u(x) dF(x) \geq \int_a^b u(x) dG(x)$ for all $u \in AP_{2,[a,b]} + \mathcal{I}_{2,[a,b]}$. Then $F \succeq_{2,[a,b]-S} G$. 
 
\end{proposition}

Note that the set $AP_{2,[a,b]} + \mathcal{I}_{2,[a,b]}$ contains the constant functions and is convex as the sum of convex sets. It is clearly also a cone. Thus, the closure of the set $AP_{2,[a,b]} +\mathcal{I}_{2,[a,b]}$ in the weak topology   $cl(AP_{2,[a,b]} +\mathcal{I}_{2,[a,b]})$ contains the maximal generator of the $2,[a,b]$-sufficient stochastic order (see Proposition \ref{Prop: maximal generator}). We summarize this result in the following Corollary. 

\begin{corollary}
    The set  $cl(AP_{2,[a,b]} +\mathcal{I}_{2,[a,b]})$ contains the maximal generator of the $2,[a,b]$-sufficient stochastic order. 
\end{corollary}

\section{Appendix: Proofs}

For the rest of the Appendix we let $\mathcal{D}_{\alph}= -\mathcal{I}_{\alph}$. We will call the stochastic order that is generated by $\mathcal{D}_{\alph}$ the $\alpha,[a,b]$-convex  stochastic order. 
That is, for two distribution functions $F$ and $G$, we say that $F$ dominates $G$ in the $\alpha,[a,b]$-convex  stochastic order, denoted by $F \succeq_{\alph-D} G$, if for every $u\in \mathcal{D}_{\alpha ,[a,b]}$ we have
		$$\int_a^b u(x) dF(x)\ge \int_a^b u(x)dG(x).$$
Note that 	$F \succeq_{\alph-D} G$ if and only if $G \succeq_{\alph-I} F$.

We first prove the following Proposition that provide some properties of  $[a,b]$-concave  functions that will be used repeatedly in the proofs of our results.

\begin{proposition} \label{prop:2}
	The following properties hold:
	\begin{enumerate}
		
			\item $\mathcal{I}_{\alpha ,[a,b]}$ is a convex cone and is closed in the pointwise topology.
		\item  Let $\beta>\alpha$, then $\mathcal{I}_{\beta,[a,b]} \subseteq \mathcal{I}_{\alpha ,[a,b]} $. 
		\item If $u\in \mathcal{I}_{\alpha ,[a,b]}$ then for every $c\in \R$, the function $g_c(x):=u(x-c)$ is in $\mathcal{I}_{\alpha,[a+c,b+c]}$.
		\item Consider $u\in \mathcal{I}_{\alpha,[a,b]}$, twice differentiable with a continuous second derivative on $[a,b]$.\footnote{The derivatives at the extreme points $a,b$ are defined by taking the left-side and right-side limits, respectively (see Definition 5.1 \cite{rudin1964principles}).} Then, $u'(b)=0$.
		\item For $\alpha>1$, the set  $ \mathcal{I}_{\alph}$ does not contain linear functions that are not constants.
	\end{enumerate}
\end{proposition}

\begin{proof}
We prove the results for $\mathcal{D}_{\alph}$ which immediately implies the results for $-\mathcal{D}_{\alph} = \mathcal{I}_{\alph}$. 

\begin{enumerate}

\item  Consider $u,v \in \mathcal{D}_{\alph}$ and $\lambda >0$. Clearly,  $u+\lambda v$ is decreasing. Notice that, 
$$(u+\lambda v) (x) - (u+\lambda)(b)=u(x)-u(b) + \lambda (v(x)-v(b))\;,$$
hence,  $(u+\lambda v) (x) - (u+\lambda)(b)$ can be written as the sum of two $\alpha$-convex function. The sum of $\alpha$-convex function is $\alpha$-convex (see the online Appendix of \cite{jensen2017distributional}). Hence, $u+\lambda v\in\mathcal{D}_{\alph}$, which shows that $\mathcal{D}_{\alph}$ is a convex cone. 

To show that $\mathcal{D}_{\alph}$ is closed under pointwise convergence consider a sequence $(u_n)$ in $\mathcal{D}_{\alph}$ such that $u_n\to u$ (pointwise). Clearly, $u$ is decreasing. The function $u(x)-u(b)$ is the limit of the $\alpha$-convex functions $u_n(x) -u_n(b)$, and hence, $u(x)-u(b)$ is $\alpha$-convex, (see the online Appendix of \cite{jensen2017distributional}). Thus, $u \in \mathcal{D}_{\alph}$. 

\item Consider $u \in \mathcal{D}_{\beta,[a,b]}$. Then $u$ is decreasing and $f(x):=(u(x)-u(b))^{\frac 1 \beta}$ is convex. Because $\beta>\alpha$, the function $g(x):=x^{\frac \beta \alpha}$ is increasing and convex. Therefore, $g(f(x))$ is convex. We conclude that, $(u(x)-u(b))^{\frac 1 \alpha}$ is convex. Thus, $u\in \mathcal{D}_{\alph}$.

\item Because $u\in  \mathcal{D}_{\alph}$ the function $g_c$ is decreasing on $[a+c,b+c]$. Take $x_1,x_2\in [a+c,b+c]$ and $\lambda\in [0,1]$.	Since $u(x)-u(b)$ is $\alpha$-convex we have 
$$	\bigg(u(\lambda (x_1-c) + (1-\lambda) (x_2-c)) - u(b)\bigg)^{\frac 1 \alpha} \le \lambda \bigg(u(x_1-c) - u(b)\bigg)^{\frac 1 \alpha} + (1-\lambda ) \bigg(u(x_2-c) - u(b)\bigg)^{\frac 1 \alpha}\;.$$
	 Since $ g_c(\lambda x_1 + (1-\lambda x_2))= u(\lambda (x_1-c) + (1-\lambda) (x_2-c))$, $g_c(b+c)= u(b)$, $g_c(x_1)=u(x_1-c)$, and $g_c(x_2)=u(x_2-c)$, we conclude that $g_c(x)-g_c(b+c)$ is $\alpha$-convex. Thus, $g_c\in \mathcal{D}_{\alpha,[a+c,b+c]}$.
	
\item Suppose for the sake of a contradiction that $u'(b)\neq 0$. Because $u'$ is continuous, a $\delta>0$ exists such that $\lim_{x\to b^-}u'(x)^2>\delta$. Notice that 
$$\lim_{x\to b^-} (u(x)-(b))u''(x)=\underbrace{\lim_{x\to b^-} (u(x)-u(b))}_{0}\underbrace{\lim_{x\to b^-} u''(x)}_{u''(b)}=0\;.$$
Thus, 
$$ \lim_{x\to b^-} \frac{(u(x) - u(b))u''(x)}{u'(x)^2} = \frac{\lim_{x\to b^-} (u(x) - u(b))u''(x)}{\lim_{x\to b^-} u'(x)^2} =0\;.$$
Because $u$ is twice differentiable with a continuous second derivative, a $\epsilon>0$ exists such that for $x\in (b-\epsilon,b)$, $\frac{(u(x) - u(b))u''(x)}{u'(x)^2} < \frac{\alpha-1}{\alpha}$. Using the $\alpha$-convex characterization for a twice differentiable function, we conclude that $u(x)-u(b)$ is not $\alpha$-convex. Therefore, $u\notin D_{\alpha,[a,b]}$ which is a contradiction. We conclude that $u'(b)=0$.

\item Let $\alpha>1$. Consider $u$ to be a linear function that is decreasing and not a constant. Notice that $u(x)-u(b)$ is twice-differentiable, and that for every $x\in[a,b]$ $u'(x)<0$ and $u''(x)=0$. We conclude that $\frac{(u(x)-u(b))u''(x)}{u'(x)^2}=0$. Thus, $u(x)-u(b)$ is not $\alpha$-convex, i.e., $u \notin \mathcal{D}_{\alph}$. 
\end{enumerate} 
\end{proof}

\subsection{Proofs of the results in Section \ref{sec:2}}

\begin{proof}[Proof of Example \ref{Example 1}] 
 Let $u \in \mathcal{I}_{\alpha  ,[a ,b]}$, $0 <\lambda  <1$ and $\alpha  \geq 1$. The $\alpha $-convexity of $u\left (b\right ) -u\left (x\right )$ implies 
\begin{align*}\left [u\left (b\right ) -u\left (\lambda a +\left (1 -\lambda \right )b\right )\right ]^{\frac{1}{\alpha }} &  \leq \lambda \left [u\left (b\right ) -u\left (a\right )\right ]^{\frac{1}{\alpha }} +\left (1 -\lambda \right )\left [u\left (b\right ) -u\left (b\right )\right ]^{\frac{1}{\alpha }} \\
 &  \Leftrightarrow \left .u\left (b\right ) -u\left (\lambda a +\left (1 -\lambda \right )b\right )\right . \leq \lambda ^{\alpha }u\left (b\right ) -\lambda ^{\alpha }u\left (a\right ) \\
 &  \Leftrightarrow \lambda ^{\alpha }u\left (a\right ) +\left (1 -\lambda ^{\alpha }\right )u\left (b\right ) \leq u\left (\lambda a +\left (1 -\lambda \right )b\right ) \\
 &  \Leftrightarrow \int _{a}^{b}u\left (x\right )dG(x) \leq \int _{a}^{b}u\left (x\right )dF(x)\end{align*}where $F$ is the distribution function of $Y$ and $G$ is the distribution function of $X$. We conclude that $Y \succeq _{\alpha  ,[a ,b] -I}X$.\footnote{Note that Example \ref{Example 1} implies that when $\alpha$ tends to infinity we have $u(b) \leq u(\lambda a + (1 -\lambda)b$. Hence, $u$ is a constant function.}
\end{proof}

\begin{proof}[Proof of Example \ref{Example 2}] 
Let $u \in \mathcal{I}_{\alpha ,[a ,b]}$, $0 <\lambda  <1$ and $\alpha  \geq 1$. From Example \ref{Example 1} we have 
\begin{equation*}u\left (x_{i}\right ) \geq \lambda _{i}^{\alpha }u(a) +\left (1 -\lambda _{i}^{\alpha }\right )u(b)
\end{equation*} 
for all $0 <\lambda  <1$. Multiplying each side of the last inequality by $p_{i}$ for $i =1 ,\ldots  ,n$ and summing the inequalities yield
\begin{align*}\sum _{i =1}^{n}p_{i}u\left (x_{i}\right ) &  \geq \sum _{i =1}^{n}\left (p_{i}\lambda _{i}^{\alpha }u(a) +p_{i}u(b) -p_{i}\lambda _{i}^{a}u\left (b\right )\right ) \\
 &  \Leftrightarrow \sum _{i =1}^{n}p_{i}u\left (x_{i}\right ) \geq \sum _{i =1}^{n}p_{i}\lambda _{i}^{\alpha }u\left (a\right ) +\left (1 -\sum _{i =1}^{n}p_{i}\lambda _{i}^{\alpha }\right )u\left (b\right ) \\
 &  \Leftrightarrow \int _{a}^{b}u(x)dF(x) \geq \int _{a}^{b}u\left (x\right )dG(x)\end{align*} where $F$ is the distribution function of $Y$ and $G$ is the distribution function of $X$. We conclude that $Y \succeq _{\alpha  ,[a ,b] -I}X$.
\end{proof}

\begin{proof}[Proof of Example \ref{Example 3}] 
From Example \ref{Example 1}, for any $u \in \mathcal{I}_{\alpha  ,[a ,b]}$ and $\alpha  \geq 1$ we have 
\begin{equation*}u\left (\lambda a +\left (1 -\lambda \right )b\right ) \geq \lambda ^{\alpha }u(a) +\left (1 -\lambda ^{\alpha }\right )u(b)
\end{equation*} 
for all $0 <\lambda  <1$. Integrating both sides yields 
\begin{align*}\int _{0}^{1}u\left (\lambda a +\left (1 -\lambda \right )b\right )d\lambda  &  \geq u(a)\int _{0}^{1}\lambda ^{\alpha }d\lambda  +u\left (b\right )\int _{0}^{1}\left (1 -\lambda ^{\alpha }\right )d\lambda  \\
 &  \Leftrightarrow \frac{1}{b -a}\int _{a}^{b}u(x)dx \geq \frac{1}{\alpha  +1}u(a) +\frac{\alpha }{a +1}u(b) \\
 &  \Leftrightarrow \int _{a}^{b}u(x)dF(x) \geq \int _{a}^{b}u\left (x\right )dG(x)\end{align*}
 where $F$ is the distribution function of $Y$ and $G$ is the distribution function of $X$. We conclude that $Y \succeq _{\alpha  ,[a ,b] -I}X$.
\end{proof}

\begin{proof}[Proof of Proposition~\ref{prop:2.5}]

1. Suppose that $F \succeq_{\alph-D}G$ and that $u\in \mathcal{D}_{\beta,[a,b]}$. Because $\beta>\alpha$, Proposition~\ref{prop:2} implies that $u\in \mathcal{D}_{\alph} $. Hence, $\E_F[u]\ge \E_G[u]$. Given that $u$ is an arbitrary function that belongs to the set  $\mathcal{D}_{\beta,[a,b]}$, we conclude that $F \succeq_{\beta,[a,b]-D}G$. 
		
2.  Consider $X \succeq_{\alph-D}Y$ and $u\in \mathcal{D}_{\alpha,[a+c,b+c]}$. Suppose that the distributions of $X$ and $Y$ are $F$ and $G$, respectively. From Proposition~\ref{prop:2} we have that $g_c(x):=u(x+c)$ belongs to the set $\mathcal{D}_{\alph} $. Hence, 
		\begin{align*}\int_a^b g_c(x)dF(x)\ge \int_a^b g_c(x)dG(x) &\iff \int_a^b u(x+c)dF(x)\ge \int_a^b u(x+c) dG(x)\\
		&\iff \int_{a+c}^{b+c}u(z)dF(z-c)\ge \int_{a+c}^{b+c}u(z)dG(z-c)\;.		
		\end{align*}
	The last equivalence comes from using the change of variables $z=x+c$. We conclude that $X+c\succeq_{\alpha,[a+c,b+c]-D}Y+c$.
		
3.  Let $b' > b$. Assume that $F\succeq_{ \alpha,[a,b']-D}G$ and $u\in \mathcal{D}_{\alph}$. We extend $u$ to the domain $[a,b']$ as follows:
		$$\hat u (x)=\begin{cases}
		u(x) &\mbox{ if } x\in[a,b]\\
		u(b)  &\mbox{ if } x\in[b,b']
		\end{cases} \;.$$
		We assert that $\hat u \in \mathcal{D}_{\alpha,[a,b']}$. Clearly, $\hat u$ is decreasing, it remains to prove that $\hat u(x)-\hat u (b')$ is $\alpha$-convex. For this extent, we claim that for $ x_1,x_2\in[a,b']$ and $\lambda\in[0,1]$ the following inequality holds:
		\begin{equation} \label{inequality: max u}
		 \bigg(\hat u(\lambda x_1 + (1-\lambda) x_2) - \hat u(b')\bigg)^{\frac 1 \alpha} \le \lambda \bigg(\hat u(x_1) - \hat u(b')\bigg)^{\frac 1 \alpha} + (1-\lambda ) \bigg(\hat u(x_2) - \hat u(b')\bigg)^{\frac 1 \alpha}\;.
		\end{equation}	
We prove this by separating our analysis in three cases:

(i) For $x_1,x_2 \in [a,b]$, we have that $\hat u(x_1)=u(x_1)$, $\hat u(x_2)=u(x_2)$, $\hat u(\lambda x_1 + (1-\lambda) x_2)= u(\lambda x_1 + (1-\lambda) x_2)$, and  $\hat u(b')= u(b)$. Thus, because $u(x)-u(b)$ is $\alpha$-convex inequality (\ref{inequality: max u}) holds. 
(ii) For $x_1,x_2\in [b,b']$, we have that  $\hat u(x_1)=\hat u(b')$, $\hat u(x_2)=\hat u(b')$, $\hat u(\lambda x_1 + (1-\lambda) x_2)= \hat u(b')$, and therefore, inequality (\ref{inequality: max u}) holds. 
(iii) The last case is when $x_1\in [a,b]$ and $x_2\in (b,b']$ (or analogously, when $x_1\in(b,b']$ and $x_2\in[a,b]$). Because $x_1 \in [a,b]$, from the first case we have that 	$$	\bigg(\hat u(\lambda x_1 + (1-\lambda) b) - \hat u(b')\bigg)^{\frac 1 \alpha} \le \lambda \bigg(\hat u(x_1) - \hat u(b')\bigg)^{\frac 1 \alpha} + (1-\lambda ) \bigg(\hat u(b) - \hat u(b')\bigg)^{\frac 1 \alpha}\;.$$
Because $\hat u$ is decreasing we have that $\hat u(\lambda x_1 + (1-\lambda) b) - \hat u(b')\ge \hat u(\lambda x_1 + (1-\lambda) x_2) - \hat u(b')$. We also have that $\hat u(b)= \hat u (x_2)$. Thus,  
	$$	\bigg(\hat u(\lambda x_1 + (1-\lambda)  x_2) - \hat u(b')\bigg)^{\frac 1 \alpha} \le \lambda \bigg(\hat u(x_1) - \hat u(b')\bigg)^{\frac 1 \alpha} + (1-\lambda ) \bigg(\hat u( x_2) - \hat u(b')\bigg)^{\frac 1 \alpha}\;.$$
	Which proves that inequality (\ref{inequality: max u}) holds.

Because $\hat u \in \mathcal{D}_{\alpha,[a,b']}$ and $F \succeq_{\alpha,[a,b']-D}G$, we have that $\int_a^{b'}\hat u (x)dF(x)\ge \int_a^{b'}\hat u (x)dG(x)$. Since $F$ and $G$ are distributions with support contained on $[a,b]$, we have that  $\int_a^{b'}\hat u (x)dF(x)= \int_a^{b}\hat u (x)dF(x)=\int_a^{b} u (x)dF(x)$ and $\int_a^{b'}\hat u (x)dG(x)= \int_a^{b}\hat u (x)dG(x)=\int_a^{b} u (x)dG(x)$. Therefore, for any $u\in\mathcal{D}_{\alph}$, we have $\E_F[u]\ge \E_G[u]$. We conclude that if $F \succeq_{\alpha,[a,b']-D} G$ then $F\succeq_{\alph-D} G  $.
\end{proof}

\begin{proof}[Proof of Proposition~\ref{prop:3}]
    From Proposition~\ref{prop:2.5} part 1 we have that if the result holds for $\alpha$ integer then it holds for every $\alpha \leq n$. Thus, in what follows we consider $\alpha= n$ for a general $n\in \mathbb{N}$.
    
    Let $u\in \mathcal{D}_{n,[a,b]}$ with $u(b)=0$. Then ${u}^{\frac{1}{n }}$ is convex. Thus, from Theorem~\ref{theo:approx} (see below), we have that $u^{\frac{1}{n }}$ may be approximated by the functions $\{c :\max \{c -x ,0\}\}$, in the sense that there exists a sequence of functions $\{u_{m}\}_{m}$ such that  \begin{equation*}u_{m}(x) =\sum \limits _{j =1}^{m}\gamma _{jm}\max \{c_{jm} -x ,0\}
	\end{equation*}
	and $u_{m}$ converges uniformly to $ u^{\frac{1}{n }}$ for some constants $\gamma _{jm} \geq 0$, $c_{jm} \in [a,b]$. We have 
	\begin{align*}\int _{a}^{b}(u_{m}(x))^{n }dG(x)
	&  =\int _{a}^{b}\left (\sum \limits _{j =1}^{m}\gamma _{jm}\max \{c_{jm} -x ,0\}\right )^{n }dG(x)\\
	&  =\int _{a}^{b}\sum _{k_{1} + . . . +k_{m} =n }\frac{n  !}{\prod \limits _{j =1}^{m}k_{j} !}\;\prod \limits _{j =1}^{m}\gamma _{jm}^{k_{j}}\max \{c_{jm} -x ,0\}^{k_{j}}dG(x)\\
	&  \leq \int _{a}^{b}\sum _{k_{1} + . . . +k_{m} =n }\frac{n  !}{\prod \limits _{j =1}^{m}k_{j} !}\;\prod \limits _{j =1}^{m}\gamma _{jm}^{k_{j}}\max \{c_{jm} -x ,0\}^{k_{j}}dF(x)%
	&  =\int _{a}^{b}(u_{m}(x))^{n }dF(x).
	\end{align*}
	The second equality follows from the multinomial theorem. The inequality follows from the fact that $F \succeq _{n,[a,b]-S}G$.  Applying the dominated convergence theorem yields
	\begin{equation*}\int _{a}^{b}u(x)dF(x) \geq \int _{a}^{b}u(x)dG(x), 
	\end{equation*} for every $u\in D_{n,[a,b]}$ with $u(b)=0$. 
	To complete the proof, take an arbitrary function $v \in D_{n,[a,b]}$.  Then $u(x):=v(x)-v(b)$ belongs to the set $D_{n,[a,b]}$ and satisfies  $u(b)=0$. Thus,
	$$ \int _{a}^{b}(v(x)-v(b))dF(x) \geq \int _{a}^{b}(v(x)-v(b))dG(x) \iff \int _{a}^{b}v(x)dF(x) \geq \int _{a}^{b}v(x)dG(x)\;,$$
	which completes the proof.
\end{proof}

We now provide a proof of a well-known result in the literature about approximation of convex and decreasing functions. 
\begin{theorem}\label{theo:approx}
	Let $u:[a,b]\to \R$ a continuous convex and decreasing function such that $u(b)=0$. Then, there is a sequence ($u _n$) of the form $u_{n}(x) =\sum \limits _{j =1}^{n}\gamma _{j}\max \{c_{j} -x ,0\}$ for some $\gamma_{j}\ge 0$ and $c_j\in [a,b]$, such that $u_m$ converges uniformly to $u$.
\end{theorem}
\begin{corollary}
	Every decreasing convex function can be approximated by decreasing continuous and convex functions.
\end{corollary}

\begin{proof}
	The proof is by construction and is based on the paper \cite{russell1989representative}. 
	
	Consider a partition of the interval of the interval $[a,b]$, $P_n=[c_n,c_{n-1},\ldots,c_0]$ such that $c_i=b-\frac i n (b-a)$ for $i=1,\ldots, n$. 
	For $i=0,\ldots,n-1$ we define
	\begin{align*}
	c_{-1}&=b\;\\
	\beta_{i}&=u(c_i)-u(c_{i-1})\\
	\gamma_{i}&=\frac 1 {c_{i}-c_{i+1}} (\beta_{(i+1)}-\beta_{i})\;.
	\end{align*}
	
	Because $c_{i-1}$ is the average point between $c_{i},c_{i-2}$, by convexity of $u$ we have that 
	$$u(c_{i})+u(c_{i-2})\ge 2 u(c_{i-1})\;,$$
	which implies that $\beta_{i}\ge \beta_{(i-1)}$ and $\gamma_i\ge 0$.
	
	Also,
	\begin{align*}
	\sum_{j=0}^i \gamma_{j}(c_j-c_{i+1})&=\sum_{j=0}^i \frac{c_{j}-c_{i+1}}{c_{j}-c_{j+1}} (\beta_{(j+1)}-\beta_{j})\\
	&=\sum_{j=0}^i (i+1-j)(\beta_{(j+1)}-\beta_{j})=-(i+1)\beta_{0}+\beta_{1}+\beta_{2}+\ldots+\beta_{(i+1)}\\
	&=\beta_{1}+\ldots+\beta_{(i+1)}=u(c_{i+1})-u({c_0})
	\end{align*}
	Because $u(c_0)=u(b)=0$, we get that
	\begin{equation}\label{a2}
	u(c_{i+1})=\sum_{j=0}^i \gamma_{j}(c_{j}-c_{i+1}) \mbox{ for every } i=0,1,\ldots, n-1\;.
	\end{equation}
	
	Define $\hat u_n(x):= \sum_{j=0}^{n-1} \gamma_{j}\max\{c_{j}-x,0\}$.  We claim that for every $\epsilon>0$ there is a sufficiently large $n$ such that for every $x\in[a,b]$ we have $|u(x)-\hat u_n(x)|<\epsilon$. Indeed, consider $x\in[a,b]$, there is $0\le k\le n-1$ such that $x\in [c_{k+1},c_k]$.
	Because $\hat u _n$ is decreasing ($\gamma_{j}$ are nonnegative), we have $\hat u _n(c_{k})\le \hat u _n(x)\le \hat u _n(c_{k+1})$. Now,
	$$\hat u _n(c_{k})= \sum_{j=0}^{n-1} \gamma_{j}\max\{c_{j}-c_{k},0\} = \sum_{j=0}^{k-1} \gamma_{j}(c_{j}-c_{k}) = u(c_k)\;,$$
	where the second equality comes from Equation~\eqref{a2}. The same argument implies that $\hat u _n(c_{k+1})=u(c_{k+1})$. Hence, for every $k=0,1,\ldots, n-1$ we have that
	\begin{equation}\label{eq:apv}
	u(c_{k})\le \hat u _n(x)\le  u(c_{k+1}) \mbox{ for every } x\in [c_{k+1},c_k]\;.\end{equation} 
	
	Because $u$ is continuous on $[a,b]$, $u$ is uniformly continuous. Thus, there is a sufficiently high $n$ such that $|u(c_{k+1}) - u(c_k)|\le \epsilon$. Second, because $u$ is decreasing we have that  $ u (c_{k})\le u (x)\le  u (c_{k+1})$. Using these two facts on inequality~\eqref{eq:apv} allow us to conclude that
	$$|u(x)-\hat u_n(x)|\le \epsilon \mbox{ for every }x\in[a,b]\;.$$
\end{proof}

\begin{proof}[Proof of Proposition~\ref{prop:4}]
From Lemma~\ref{lemma:0} (see below), we have that for  $c_1,c_2\in [a,b]$ with  $c_2\ge c_1$ the expression
	$$ \int_a^{b}  \max\{c_1-x,0\}\max\{c_2-x,0\} dF(x) - \int_a^{b}  \max\{c_1-x,0\}\max\{c_2-x,0\} dG(x) $$ is equal to
	$$(c_2-c_1)\bigg[\int_0^{c_1} F(x) dx -\int_0^{c_1} G(x) dx  \bigg] + 2\int_0^{c_1} \bigg(\int_0^x F(z)dz-\int_0^x G(z)dz\bigg)dx \ge 0\;.$$
	
	Because the above inequality is linear in $c_2$, we have that it holds for every $c_2\in[c_1,b]$ if and only if it holds for $c_2=b$ and for $c_2=c_1$. Evaluating it at these two points we obtain the first and the second inequalities of Proposition~\ref{prop:4}, respectively.
\end{proof}

\begin{lemma}\label{lemma:0}
	Consider a distribution $F$ on $[a,b]$. For every $c_1\le c_2$ in $[a,b]$ we have that 
	$$\int_a^{b}  \max\{c_1-x,0\}\max\{c_2-x,0\} dF(x)= (c_2-c_1)\int_a^{c_1} F(x)dx + 2 \int_a^{c_1}\int_a^{x} F(z)dzdx\;.$$
\end{lemma}

\begin{proof}
	Because $c_1\le c_2$ we have that
	\begin{align}
	\int_a^{b}  \max\{c_1-x,0\}\max\{c_2-x,0\} dF(x) &= \int_a^{c_1}(c_1-x)(c_2-x)dF(x)\\
	&=c_2\int_a^{c_1}(c_1-x)dF(x)-\int_a^{c_1} x(c_1-x)dF(x)\;. \label{eq:0}
	\end{align}
	Using integration by parts for Lebesgue-Stieltjes integrals, we have that
	\begin{equation}\label{eq:aux1}
	\int_a^{c_1}(c_1-x)dF(x)= (c_1-x)F(x)\Big|_{a^-}^{c_1^+}+ \int_a^{c_1} F(x)dx = \int_a^{c_1} F(x)dx \;,
	\end{equation}
	where the second equality comes from $F(a^-)=0$.
	
	To tackle the second term in Equation~\eqref{eq:0}, define $ v(x):=\int_a^{x} (c_1-z)dF(z)$ for $x\in [a,c_1]$. Using integration by parts and the fact that $F(a^-)=0$, we have that $v(x)=(c_1-x)F(x)+\int_a^{x}F(z)dz$. Define $u(x)=x$. We have that $\int_a^{c_1} x(c_1-x)dF(x)=\int_a^{c_1}u(x) dv(x)$. Using integration by parts and the fact that $v(a^-)=0$, we obtain
	$$\int_a^{c_1} x(c_1-x)dF(x)=\int_a^{c_1} x F(x) dx - \int_a^{c_1}\int_a^{x} F(z)dzdx\;.$$
	
	Once again, using integration by parts, we have that $\int_a^{c_1} x F(x) dx =c_1\int_a^{c_1}F(x)dx-\int_a^{c_1}\int_a^{x} F(z)dzdx$.
	Thus,
	\begin{equation}\label{eq:aux2}
	\int_a^{c_1} x(c_1-x)dF(x)=c_1\int_a^{c_1}F(x)dx- 2 \int_a^{c_1}\int_a^{x} F(z)dzdx\;.
	\end{equation}
	
	Therefore, plugging \eqref{eq:aux1} and \eqref{eq:aux2} into Equation~\eqref{eq:0} we get that
	\begin{equation*}
	 \int_a^{c_1}(c_1-x)(c_2-x)dF(x) = (c_2-c_1)\int_a^{c_1} F(x)dx + 2 \int_a^{c_1}\int_a^{x} F(z)dzdx\;.
	 \end{equation*}
	\end{proof}

\subsection{Proofs of Section \ref{sec:3}}

\begin{proof}[Proof of Proposition~\ref{prop:consumption-savings}]
We prove only part (i) (using Proposition \ref{prop:2.5} the proof of part (ii) is identical to the proof of part (i)). 

Define the function $g_{s} :[0 ,\overline{y}] \rightarrow \mathbb{R}_{ +}$ by  $g_{s}(y) : =u^{ \prime }(Rs +y)$ for all $0 \leq s \leq x$. First note that $g_{s}(y)$ is a $2 ,[0 ,Rx +\overline{y} -Rs]$-convex function.   
 To see this, note that \begin{equation*}g_{s}(y) -g_{s}(Rx +\overline{y} -Rs) =u^{ \prime }\left (Rs +y\right ) -u^{ \prime }(Rx +\overline{y})
\end{equation*} 
is $2$-convex because $u^{ \prime }$ is $2 ,[0 ,Rx +\overline{y}]$-convex and $0 \leq Rs +y \leq Rx +\overline{y}$ for $0 \leq y \leq Rx -Rs +\bar{y}$.

From Lemma \ref{Lemma: suff prop} (see below), $F  \succeq _{2 ,[0 ,Rx +\overline{y}] -S} G$ implies that $F \succeq _{2 ,[0 ,Rx -Rs +\overline{y}] -S}G$ for all $s \in [0 ,x]$.  Let $h_{s} (s ,q)$ be the derivative of $h$ with respect to $s$. Let $s \in [0 ,x]$. We have 
\begin{align*}h_{s}(s ,F) &  = -u^{ \prime }\left (x -s\right ) +\beta \int _{0}^{\overline{y}}u^{ \prime }(Rs +y)dF(y) \\
 &  = -u^{ \prime }(x -s) +\int _{0}^{\overline{y}}g_{s}\left (y\right )dF(y) \\
 &  \leq  -u^{ \prime }\left (x -s\right ) +\int _{0}^{\overline{y}}g_{s}(y)dG(y) =h_{s}(s ,G) ,\end{align*}where the inequality follows from the facts that $F \succeq _{2 ,[0 ,Rx -Rs +\overline{y}] -S}G$ and that $g_{s}(y)$ is $2 ,[0 ,Rx -Rs +\bar{y}]$ convex. Theorem 6.1 in \cite{topkis1978}
implies that $g(G) \geq g(F)$.           
\end{proof}

\begin{lemma}
\label{Lemma: suff prop}Let $F$ and $G$ be two distributions. Suppose that $F \succeq _{2 ,[a ,b] -S}G$. Then $F \succeq _{2 ,[a ,b^{ \prime }] -S}G$ for all $b^{ \prime } \in (a ,b)$. 
\end{lemma}

\begin{proof}
Assume that $F \succeq _{2 ,[a ,b] -S}G$. Let $b^{ \prime } \in (a ,b)$ and $c \in [a ,b^{ \prime }]$.  

Note that
$F \succeq _{2 ,[a ,b] -S}G$ implies $\int _{a}^{c}\left (\int _{a}^{x}(F(z) -G(z))dz\right )dx \geq 0$. Thus, condition (\ref{ine:2}) holds.

If $\int _{a}^{c}(F(x) -G(x))dx \geq 0$ then \begin{equation*}(b^{ \prime } -c)\left [\int _{a}^{c}(F(x) -G(x))dx\right ] +2\int _{a}^{c}\left (\int _{a}^{x}(F(z) -G(z))dz\right )dx \geq 0.
\end{equation*}If $\int _{a}^{c}(F(x) -G(x))dx <0$ then 
\begin{align*}(b^{ \prime } -c)\left [\int _{a}^{c}(F(x) -G(x))dx\right ] +2\int _{a}^{c}\left (\int _{a}^{x}(F(z) -G(z))dz\right )dx \\
 \geq (b -c)\left [\int _{a}^{c}(F(x) -G(x))dx\right ] +2\int _{a}^{c}\left (\int _{a}^{x}(F(z) -G(z))dz\right )dx \geq 0\end{align*}
 where the last inequality follows because $F \succeq _{2 ,[a ,b] -S}G$. So condition (\ref{ine:1}) holds. 

We conclude that condition (\ref{ine:1}) and (\ref{ine:2}) hold for all $c \in [a ,b^{ \prime }]$. Thus, $F \succeq _{2 ,[a ,b^{ \prime }]-S}G$.
\end{proof}

\begin{proof}[Proof of Proposition \ref{Prop:Self-protection}]
The proof follows immediately from  Lemma \ref{Lemma:two-point} below. 
\end{proof}
\begin{lemma}\label{Lemma:two-point}
     Suppose that $X$ yields $x_{1}$ with probability $p$ and $x_{3}$ with probability $1-p$. $Y$ yields $x_{2}$ with probability $q$ and $x_{4}$ with probability $1-q$. 
     
     Suppose that the expected value of $X$ is higher than the expected value of $Y$, i.e., 
     \begin{equation}\label{Ineq:two-point-expectation}
         px_{1}+(1-p)x_{3} \geq  qx_{2}+(1-q)x_{4}. 
     \end{equation}
     
     Then $X\succeq_{2,[x_{1},x_{4}]-D} Y$ if and only if
     \begin{equation}p(x_{4}-x_{1})^{2} +(1 -p)(x_{4}-x_{3})^{2} \geq q(x_{4} -x_{2})^{2}. \label{Ineq:two-point-iff}
\end{equation}

\end{lemma}

\begin{proof}[Proof of Lemma~\ref{Lemma:two-point}]
Let $F$ be the distribution function of $X$ and let $G$ be the distribution function of $Y$. Let $c \in [x_{1} ,x_{4}]$. 

In Step 1 we show that condition~\eqref{ine:2} holds if and only if inequality (\ref{Ineq:two-point-iff}) holds. In Step 2 we show that if condition~\eqref{ine:2} holds then condition~\eqref{ine:1} also holds. Thus, from Corollary \ref{coro:1} inequality (\ref{Ineq:two-point-iff}) holds  if and only if $F\succeq_{2,[a,b]-D} G$.

\textbf{Step 1.} Condition~\eqref{ine:2} holds if and only if inequality (\ref{Ineq:two-point-iff}) holds. We consider two cases.

Case 1.  $x_{1} \leq c \leq x_{3}$. If $c \leq x_{2}$ condition~\eqref{ine:2} trivially holds. Suppose that $c >x_{2}$.  

Note that $\int _{a}^{c}\max (c -x ,0)^{2}dF(x) =p(c -x_{1})^{2}$ and  $\int _{a}^{c}\max (c -x ,0)^{2}dG(x) =q(c -x_{2})^{2}$. Thus, condition~\eqref{ine:2} holds if $\sqrt{p}(c -x_{1}) \geq \sqrt{q}(c -x_{2})$ for all $x_{2} \leq c \leq x_{3}$. The last inequality is linear in $c$ and clearly holds for $c =x_{2}$. So  it holds for all $x_{2} \leq c \leq x_{3}$ if it holds for $c =x_{3}$, i.e., the following inequality holds:
\begin{equation}\sqrt{p}(x_{3} -x_{1}) \geq \sqrt{q}(x_{3} -x_{2}) . \label{Lemma:2 in.1}
\end{equation}

Case 2. $x_{3} \leq c \leq x_{4}$. In this case  $\int _{a}^{c}\max (c -x ,0)^{2}dF(x) =p(c -x_{1})^{2} +(1 -p)(c -x_{3})^{2}$ and  $\int _{a}^{c}\max (c -x ,0)^{2}dG(x) =q(c -x_{2})^{2}$. 

Thus, condition~\eqref{ine:2} holds if    \begin{equation}p(c-x_{1})^{2} +(1 -p)(c -x_{3})^{2} \geq q(c-x_{2})^{2} \label{Lemma:2 in.2}
\end{equation}
for all $x_{3} \leq c \leq x_{4}$. Clearly, inequality (\ref{Lemma:2 in.2}) with $c =x_{3}$ is the same as inequality (\ref{Lemma:2 in.1}), so inequality (\ref{Lemma:2 in.2}) holds for all $x_{3} \leq c \leq x_{4}$ if and only if condition~\eqref{ine:2} holds.

 Consider the convex optimization problem \begin{equation*}\min _{x_{3} \leq c \leq x_{4}}k(c) : =p(c -x_{1})^{2} +(1 -p)(c -x_{3})^{2} -q(c -x_{2})^{2} .
\end{equation*}Note that $k^{ \prime }(x_{4}) \leq 0$ if and only if $(1 -q)x_{4} +qx_{2} \leq px_{1} +(1 -p)x_{3}$ which holds from our assumption (see inequality (\ref{Ineq:two-point-expectation})). Because $k$ is convex, $k^{ \prime }$ is increasing on $[x_{3} ,x_{4}]$, so $k^{ \prime }(c) \leq 0$ for all $x_{3} \leq c \leq x_{4}$. Thus, the optimal solution for the optimization problem $\min _{x_{3} \leq c \leq x_{4}}k(c)$ is $c =x_{4}$. 

This implies that inequality (\ref{Lemma:2 in.2}) holds for all $x_{3} \leq c \leq x_{4}$ if and only if $k(x_{4}) \geq 0$, i.e., 
\begin{equation}p(x_{4}-x_{1})^{2} +(1 -p)(x_{4}-x_{3})^{2} \geq q(x_{4} -x_{2})^{2}. \label{Lemma:2 in.3}
\end{equation}
We conclude that condition~\eqref{ine:2} holds if and only if inequality (\ref{Lemma:2 in.3}) holds.

{\bf Step 2.} Condition~\eqref{ine:2} implies condition~\eqref{ine:1}. We again consider two cases.

Case 1.  $x_{1} \leq c \leq x_{3}$. If $c \leq x_{2}$ condition~\eqref{ine:1}  trivially holds. Suppose that $c >x_{2}$. 

Note that $\int _{a}^{c}\max \{c -x ,0\}\max \{x_{4} -x ,0\}dF(x) =p(c -x_{1})(x_{4} -x_{1})$ and $\int _{a}^{c}\max (c -x ,0)\max (x_{4}-x ,0)dG(x) =q(c -x_{2})(x_{4} -x_{2})$. Thus, condition~\eqref{ine:1} holds if $p(c -x_{1})(x_{4} -x_{1}) \geq q(c -x_{2})(x_{4} -x_{2})$ for all $x_{2} \leq c \leq x_{3}$. The last inequality is linear in $c$ and clearly holds for $c =x_{2}$. So it holds for all $x_{2} \leq c \leq x_{3}$ if it holds for $c =x_{3}$, i.e., the following inequality holds:  \begin{equation}p(x_{3} -x_{1})(x_{4} -x_{1}) \geq q(x_{3} -x_{2})(x_{4} -x_{2}) . \label{Lemma:2 in.4}
\end{equation}

Case 2. $x_{3} \leq c \leq x_{4}$. In this case,  
\begin{equation*}\int _{a}^{c}\max (c -x ,0)\max (x_{4}-x ,0)dF(x) =p(c -x_{1})(x_{4}-x_{1}) +(1-p)(c -x_{3})(x_{4}-x_{3})
\end{equation*}

and $\int _{a}^{c}\max (c-x ,0)\max (x_{4} -x ,0)dG(x) =q(c -x_{2})(x_{4} -x_{2})$. Thus, condition~\eqref{ine:1} holds if \begin{equation*}w(c) : =p(c -x_{1})(x_{4} -x_{1}) +(1 -p)(c -x_{3})(x_{4} -x_{3}) -q(c -x_{2})(x_{4} -x_{2}) \geq 0
\end{equation*} for all $x_{3} \leq c \leq x_{4}$. Because $w(c)$ linear in $c$ it is enough to check for $c =x_{3}$ and $c =x_{4}$ to verify that $w(c) \geq 0$ holds for all $x_{3} \leq c \leq x_{4}$. Note that  
\begin{align*}w^{ \prime }(c) &=p(x_{4} -x_{1}) +(1 -p)(x_{4} -x_{3}) -q(x_{4} -x_{2}) \\
 &=qx_{2} +(1 -q)x_{4} -px_{1} -(1 -p)x_{3} \leq 0\end{align*}
 where the inequality follows from our assumption. Thus, if $w(x_{4}) \geq 0$ then $w(c) \geq 0$ holds for all $x_{3} \leq c \leq x_{4}$. Inequality (\ref{Lemma:2 in.3}) implies that $w(x_{4}) \geq 0$ so $w(x_{3}) \geq 0$, i.e., inequality (\ref{Lemma:2 in.4}) holds. 

Now note that inequality (\ref{Lemma:2 in.3}) holds if and only if $w(x_{4}) \geq 0$. We conclude that condition~\eqref{ine:2} implies condition~\eqref{ine:1}. 
\end{proof}

\begin{proof}
[Proof of Proposition~\ref{Prop: game}] 
Recall that a Bayesian Nash equilibrium (BNE) of the game is given by $(e_{1}^{ \ast } ,e_{2}^{ \ast } (\theta ))$ where
\begin{equation*}e_{1}^{ \ast } =\underset{e_{1} \in E}{\ensuremath{\operatorname*{argmax}}}\int _{0}^{1}e_{1} e_{2}^{ \ast } (\theta ) d F (\theta ) -c_{1} (e_{1})
\end{equation*}and
\begin{equation*}e_{2}^{ \ast } (\theta ) =\underset{e_{2} \in E}{\ensuremath{\operatorname*{argmax}}} \ e_{1}^{ \ast } e_{2} -\frac{e_{2}^{k +1}}{(k +1)(1 -\theta )^{l}} ,\text{\ for\ }\theta  \in [0 ,1).
\end{equation*}
The proof proceed with the following steps.

\textbf{Step 1}. $e_{2} (\theta ) =\underset{e_{2} \in E}{\ensuremath{\operatorname*{argmax}}}\ e_{1} e_{2} - \frac{e_{2}^{k +1}}{(k+1)(1 -\theta )^{l}}$ is decreasing and $\alpha ,[0 ,1]$-convex. 
Let $h(e_{2}) =e_{1} e_{2} - \frac{e_{2}^{k +1}}{(k+1)(1 -\theta )^{l}}$ . It is easy to see that $h$ is strictly concave. Because $\theta  \in [0 ,1)$, we have $h^{ \prime }\left (1\right ) =e_{1} -\frac{1}{\left (1 -\theta \right )^{l}} \leq 0$ for all $e_{1} \in E$. In addition $h^{ \prime } (0) =e_{1} \geq 0$ for all $e_{1} \in E$. 

We conclude that the first order condition $h^{ \prime} (e_{2}) =0$ holds for all for all $e_{1} \in E$. The first order condition implies that $e_{1} -\frac{e_{2}^{k}}{\left (1 -\theta \right )^{l}} =0$. Thus, $e_{2} (\theta ) =e_{1}^{1/k}\left (1 -\theta \right )^{l/k}$ is a decreasing and an $\alpha ,[0 ,1]$-convex function when $l \geq \alpha k$. 

\textbf{Step 2.} Denote by $ \Delta ([0 ,1])$ the set of all distributions over $[0 ,1]$. Define the operator $y :E \times  \Delta ([0 ,1]) \rightarrow E$ by
\begin{align*}y(e,F) =\underset{e_{1} \in E}{\ensuremath{\operatorname*{argmax}}} & \int _{0}^{1} (e_{1} \tilde{e}_{2} (\theta  ,e) -c_{1} (e_{1}))d F (\theta ) \\
 & \text{s.t.\ }\tilde{e}_{2} (\theta  ,e) =\underset{e_{2} \in E}{\ensuremath{\operatorname*{argmax}}}\ e e_{2} - \frac{e_{2}^{k +1}}{\left (k +1\right )\left (1 -\theta \right )^{l}}\text{.}\end{align*}
 We now show that the operator $y$ is increasing on $E \times  \Delta ([0 ,1])$ where $ \Delta ([0 ,1])$ is endowed with the $\alpha  ,[0 ,1]$-convex stochastic order, i.e., $y (e^{ \prime } ,F^{ \prime }) \geq y (e ,F)$ for all $e^{ \prime } \geq e$ and $F^{ \prime }  \succeq _{\alpha  ,[0 ,1] -D}F$.
 
 Suppose that $e^{ \prime } \geq e$ and fix $F \in  \Delta ([0 ,1])\text{.}$ Since $\tilde{e}_{2} (\theta  ,e)$ is increasing in $e$ for all $\theta  \in [0 ,1)$ (this follows from a standard comparative statics argument, see \cite{topkis1978}),
we have $\int _{0}^{1}\tilde{e}_{2} (\theta  ,e^{ \prime })d F (\theta ) \geq \int _{0}^{1}\tilde{e}_{2} (\theta  ,e)d F (\theta )$, which implies that $y (e ,F) \geq y (e^{ \prime } ,F)$. 

Now suppose that $F^{ \prime }  \succeq _{\alpha  ,[0 ,1] -D }F$, and fix $e \in E$. From Step 1, $\tilde{e}_{2} (\theta  ,e)$ is $\alpha  ,[0 ,1]$-convex and decreasing. Thus,  $\int _{0}^{1}\tilde{e}_{2} (\theta  ,e^{ \prime })d F^{ \prime } (\theta ) \geq \int _{0}^{1}\tilde{e}_{2} (\theta  ,e)d F (\theta )$, which implies that $y (e ,F^{ \prime }) \geq y (e,F)$. 

\textbf{Step 3.} From Step 2, $y :E \times  \Delta ([0 ,1]) \rightarrow E$ is an increasing map from the complete lattice $E$ into $E$. From Corollary 2.5.2 in \cite{topkis2011supermodularity},
the greatest fixed point of $y$ exists and is increasing in $F$ on $ \Delta ([0 ,1])$. 

Let $\bar{e}_{1} (F) =y (\bar{e}_{1} (F) ,F)$ be the greatest fixed point of $y$. Let $(\bar{e}_{1} (F) ,\bar{e}_{2} (\theta  ,F))$ be the corresponding BNE, i.e., $\bar{e}_{1} (F) =y (\bar{e}_{1} (F) ,F)$ and $\bar{e}_{2} (\theta  ,F)=\tilde{e}_{2} (\theta ,\bar{e}_{1} (F))$. Thus, if $F^{ \prime }  \succeq _{\alpha ,[0 ,1]-D}F$ we have 
\begin{equation*}m(F^{ \prime }) =\bar{e}_{1} (F^{ \prime }) \bar{e}_{2} (\theta  ,F^{ \prime }) =\bar{e}_{1} (F^{ \prime }) \tilde{e}_{2} (\theta  ,\bar{e}_{1} (F^{ \prime })) \geq \bar{e}_{1} (F) \tilde{e}_{2} (\theta , \bar{e}_{1} (F^{ \prime })) \geq \bar{e}_{1} (F) \tilde{e}_{2} (\theta  ,\bar{e}_{1} (F)) =m (F).
\end{equation*}
The first inequality follows from the fact that the greatest fixed point of $y$ is increasing in $F$. The second inequality follows from the fact that $\tilde{e}_{2}(\theta  ,e)$ is increasing in $e$.

This concludes the proof of the Proposition.
\end{proof}

\begin{proof}[Proof of Lemma~\ref{Lemma:1}]
The proof has two steps. We first show that inequality~\eqref{ine:dist} is a necessary and sufficient condition for condition~\eqref{ine:2} to hold. We next prove that it also implies condition~\eqref{ine:1}. From Corollary~\ref{coro:1}, we conclude that $F\succeq_{2,[a_1,b_1]-D} G$.

Before heading to the proof, by simple algebraic manipulations we obtain that
	$$ \int_{a_1}^c F(x)dx=	\frac{(c-a_1)^2}{2(b_1-a_1)}
	\; \mbox{ and } \; {
		\int_{a_1}^c G(x)dx= \begin{cases}
		0 &\mbox{ if } c\in[a_1,a_2)\\
		\frac{(c-a_2)^2}{2(b_2-a_2)} &\mbox{ if } c\in [a_2,b_2)\\
		c-\frac{a_2+b_2}2 &\mbox{ if } c\in[b_2,b_1]
		\end{cases}\;.} $$
	And similarly,
	$$ \int_{a_1}^c\int_{a_1}^x F(z)dzdx=\frac{(c-a_1)^3}{6(b_1-a_1)}
	\; \mbox{ and } \; {
		\int_{a_1}^c\int_{a_1}^x G(z)dzdx= \begin{cases}
		0 &\mbox{ if } c\in[a_1,a_2)\\
		\frac{(c-a_2)^3}{6(b_2-a_2)} &\mbox{ if } c\in [a_2,b_2)\\
		\frac{(b_2-a_2)^2}{6}+\frac{(c-a_2)(c-b_2)}{2} &\mbox{ if } c\in[b_2,b_1]
		\end{cases}\;.} $$ 

\textbf{Step 1.}	
Define $h(c):= \int_{a_1}^c\int_{a_1}^x F(z) - G(z) dzdx$, we look for $(a_1,b_1,a_2,b_2)$ for which  $h$ is non-negative on $[a_1,b_1]$. We separate our analysis in the following subintervals of $[a_1,b_2]$:
	\begin{itemize}
		\item 	For $[a_1,a_2]$,  clearly $h(c)$ is non-negative (independent of the parameters).
		
		\item For $(a_2,b_2)$, we claim that $h$ does not have any local minimum. To see this, suppose by contradiction that there is such a minimum $c^*$. Then, because $h$ is twice differentiable we must have that $h'(c^*)=0$ and $h''(c^*)\ge 0$. This two conditions are mutually impossible: 
		\begin{align*}
	h'(c^*)=0\iff	\frac{(c^*-a_1)^2}{2(b_1-a_1)} -  \frac{(c^*-a_2)^2}{2(b_2-a_2)} =0\;, \end{align*}
		dividing the equation, in each side, by $\frac{(c^*-a_1)}2$ we have that 
		$$ \frac{c^*-a_1}{b_1-a_1} -  \frac{c^*-a_2}{c^*-a_1}\frac{ c^*-a_2}{b_2-a_2} =0 \;.$$
		Because $a_1<a_2$ and $c^*\in (a_1,a_2)$, we have that  $ \frac{c^*-a_2}{c^*-a_1} < 1$. Hence, 
		$$  \frac{c^*-a_1}{b_1-a_1} - \frac{c^*-a_2}{b_2-a_2} <0 \iff h''(c^*)<0 \;.$$
We conclude that $h$ does not have a local minimum on $(a_2,b_2)$.
		\item For $[b_2,b_1]$, we claim that $h$ is strictly concave. Indeed, simple computations lead us to $h''(c) = -1 + \frac{c-a_1}{b_1-a_1}$, which is negative for $c\in (b_2,b_1)$. By the concavity of $h$, we have that $h\ge 0$ if and only if $h(b_2)$ and $h(b_1)$ are positive. 
	
		Suppose that $h'(b_2)\le 0$. By the concavity of $h$ we have that $h$ is decreasing over $[b_2,b_1]$. Thus $h\ge 0$ over $[b_2,b_1] $ if and only if $h(b_1)\ge 0$.
		
		 We assert that if $h'(b_2)>0$ then $h(b_2)\ge 0$. Suppose for the sake of contradiction that $h(b_2)<0$. Because $h(b_2)<0$, $h'(b_2)>0$ and $h(a_2)>0$, a local minimum exists over $(a_2,b_2)$. This contradicts the second bullet. Hence if $h'(b_2)>0$, then a necessary and sufficient condition for $h\ge0$ over $[b_2,b_1]$ is that $h(b_1)\ge 0$.
		
		We conclude that $h\ge 0$ over $[b_2,b_1] $ if and only if $h(b_1)\ge 0$. 

	\end{itemize}
	From the above discussion, we conclude that $h(c)\ge 0$ on $[a_1,b_1]$ if and only if $h(b_1)\ge 0$. Thus, condition~\eqref{ine:2} holds if and only if
	\begin{align}
	(b_1-a_1)^2 -3 (b_1-b_2)(b_1-a_2) \ge (b_2-a_2)^2 \label{ine_unif_2}  \;.
	\end{align}
	Solving  for $b_1$ we have that
	\begin{align*} 
	b_1&\ge \frac{3(a_2+b_2)-2a_1 - \sqrt{a_2^2+10a_2b_2+b_2^2-12a_1(a_2+b_2-a_1)}}4\\
	b_1&\le \frac{3(a_2+b_2)-2a_1 + \sqrt{a_2^2+10a_2b_2+b_2^2-12a_1(a_2+b_2-a_1)}}4\;.
	\end{align*}
	From the Lemma's assumption we have that $b_1> b_2+a_2-a_1$. We assert that this implies that  the first inequality always holds.  Indeed, observe that
\begin{align*}
& b_2+a_2-a_1 - \frac{3(a_2+b_2)-2a_1 - \sqrt{a_2^2+10a_2b_2+b_2^2-12a_1(a_2+b_2-a_1)}}4\\& \quad >  b_2+a_2-a_1 - \frac{3(a_2+b_2)-2a_1 }4
=\frac{b_2+a_2-2a_1}4 >0\;,
\end{align*}
where the last inequality follows from $a_1<a_2<b_2$. Therefore, condition~\eqref{ine:2} holds if and only inequality~\eqref{ine:dist} holds.

\textbf{Step 2.}
 We show that condition (\ref{ine:2}) implies condition (\ref{ine:1}). Define $g(c):= (b_1-c) \int_{a_1}^c F(x) - G(x) dx +2 (\int_{a_1}^c\int_{a_1}^x F(z) - G(z) dzdx)$. Similar to the first step, we separate our analysis in the following subintervals of $[a_1,b_2]$:

\begin{itemize}
	\item For $[a_1,a_2]$, trivially, $g$ is non-negative.
	\item For $[b_2,b_1]$, we claim that $g$ is strictly convex. Indeed, $g''(c)=\frac {b_1-c}{b_1-a_1}>0$. Because $g'(b_1)= \frac{b_2+a_2-(b_1-a_1)}{2}$, which is strictly negative by the Lemma's assumption, we have that $g$ is decreasing on $[b_2,b_1]$. Therefore, $g\ge 0$ on $[b_2,b_1]$ if and only if $0\le g(b_1)=h(b_1)$.
	
	\item For the case $(a_2,b_2)$, we claim that $g$ does not have a local minimum. To prove this, we show that $g$ is concave. Indeed, for $c\in (a_2,b_2)$ we have that $g''(c)= (b_1-c)(\frac 1{b_1-a_1}- \frac 1{b_2-a_2})<0$. 
\end{itemize}
From the above discussion we conclude that condition~\eqref{ine:1} holds if and only if $h(b_2)\ge 0$ which is equivalent to inequality~\eqref{ine:dist}.
\end{proof}

\begin{proof}[Proof of Proposition~\ref{prop:HH ineq}] 
For $t =1/3$ the right-hand-side inequality follows from Example 3. For  $t \geq \frac{1}{3}$  the right-hand-side inequality follows from the fact that $f$ is decreasing.  

We now prove the left-hand-side of the inequality. Let $f \in \mathcal{D}_{2 ,[a ,b]}$ and $a <b$.  

From Lemma \ref{Lemma:1} we have \begin{equation}\frac{1}{b -a}\int _{a}^{b}f(x) \geq \frac{1}{b_{n} -a_{n}}\int _{a_{n}}^{b_{n}}f\left (x\right )dx \label{Ineq: H pf1}
\end{equation}for all $(a_{n} ,b_{n})$ such that \begin{equation}4b \leq 3\left (a_{n} +b_{n}\right ) -2a +\sqrt{a_{n}^{2} +10a_{n}b_{n} +b_{n}^{2} -12a\left (a_{n} +b_{n} -a\right )} \label{Ineq: H pf2}
\end{equation}
and $a <a_{n} <b_{n} <b$. Now suppose that $(a_{n} ,b_{n})_{n =1}^{\infty }$ is a sequence of numbers such that $a_{n} \rightarrow \theta $ and $b_{n} \rightarrow \theta $, and inequality (\ref{Ineq: H pf2}) and the inequalities $a <a_{n} <b_{n} <b$ hold for all $n$. We have \begin{equation}\frac{1}{b -a}\int _{a}^{b}f\left (x\right )dx \geq \lim _{n \rightarrow \infty }\frac{1}{b_{n} -a_{n}}\int _{a_{n}}^{b_{n}}f(x)dx =\lim _{n \rightarrow \infty }\frac{1}{b_{n} -a_{n}}f(\zeta _{n})(b_{n} -a_{n}) =f\left (\theta \right ). \label{Ineq: H pf3}
\end{equation}
The first equality follows from the mean value theorem for integrals (note that $f$ is continuous on $[a_{n} ,b_{n}]$ because it is convex on $[a ,b]$). The second equality follows since $\zeta _{n} \in \left (a_{n} ,b_{n}\right )$ for all $n$.

Let $0 <\lambda  <1$ be such that $\theta  =\lambda b +(1 -\lambda )a$. Suppose that  $0 <\lambda  <1$ is chosen such that inequality (\ref{Ineq: H pf2}) holds as equality when $a_{n} \rightarrow \theta $ and $b_{n} \rightarrow \theta$. We have 
\begin{align*}4b &  =6\theta -2a +\sqrt{12 \theta^{2} -12a(2\theta -a)} \\
 & = 6\left (\lambda b +\left (1 -\lambda \right )a\right ) -2a +\sqrt{12\left (\lambda b +\left (1 -\lambda \right )a\right )^{2} -12a(2(\lambda b +\left (1 -\lambda \right )a) -a)} \\
 &  \Leftrightarrow 4b -4a =6\lambda \left (b -a\right ) +\sqrt{12(\lambda b +\left (1 -\lambda \right )a -a)^{2}} \\
 &  \Leftrightarrow 2b -2a =3\lambda \left (b -a\right ) +\sqrt{3}\lambda \left (b -a\right ) \\
 &  \Leftrightarrow \lambda  =\frac{2}{3 +\sqrt{3}} .\end{align*}From inequality (\ref{Ineq: H pf3}) and the fact that $f$ is decreasing we have \begin{equation*}f\left (\gamma b +\left (1 -\gamma \right )a\right ) \leq f\left (\frac{2}{3 +\sqrt{3}}b +\left (1 -\frac{2}{3 +\sqrt{3}} \right )a\right ) \leq \frac{1}{b -a}\int _{a}^{b}f\left (x\right )dx
\end{equation*}
for all $\gamma  \geq \frac{2}{3 +\sqrt{3}}$. This completes the proof of the Proposition.     
\end{proof}

\subsection{Proof of Proposition \ref{prop:alpha-maximal}}

\begin{proof}[Proof of Proposition~\ref{prop:alpha-maximal}]
Suppose that $F \succeq_{\alpha-DCX} G$. We proceed with the following steps. 

	\textbf{Step 1} We assert that if $u:[a,b]\to \R_+$  is decreasing, nonnegative, twice differentiable, with $u''>0$ (i.e., $u$ is strictly convex), then there exists a $C>0$ large enough such that $u+C$ is alpha convex. 

From compactness there exists an $\epsilon >0$ such that for every $x\in [a,b]$, we have $u''(x)>\epsilon$. Let $M=\frac {\alpha-1}{\alpha}\max_{\{x\in [a,b]\}} u'(x)^2$. Because $u$ is nonnegative, we have that
	$$\big(u(x) +\underbrace{\frac{M}{\epsilon}}_{C}\big) u''(x) \ge u'(x)^2 \frac {\alpha-1}{\alpha} \mbox{ for every }x\in [a,b] \;.$$
	We conclude that the function $\tilde u := u +C$ is an $\alpha$-convex function.
	
	\textbf{Step 2} We assert that if $F\succeq_{\alpha-DCX}G$, then for every decreasing, twice differentiable, and strictly convex function $u:[a,b]\to \R_+$ we have
	$$\int_a^b u(x) dF(x) \ge \int_a^b u(x)  dG(x)\;.$$
	Let $u:[a,b]\to \R_+$ be decreasing, twice differentiable, and strictly convex. From Step 1, there exists a $C>0$ such that $u+C$ is $\alpha$-convex. Therefore, if $F\succeq_{\alpha-DCX}G$ we have 
	$$\int_a^b u(x) +C dF(x) \ge \int_a^b u(x) +C dG(x) \iff 	\int_a^b u(x) dF(x) \ge \int_a^b u(x)  dG(x)\;.$$
	
	\textbf{Step 3} Define $u_n(x)= u(x) + \frac 1 n (x-b)^2$. Clearly $u_n$ is decreasing, twice differentiable, and $u_n''(x)=u''(x) +\frac {2}{n}\ge\frac {2}{n}$, where the last inequality holds because $u$ is convex. Because the sequence $u_n$ is bounded and converges pointwise to $u$. From the dominated convergence theorem we get that 
	$$\lim \int_a^b u_n(x) dF(x) = \int_a^b u(x) dF(x) \mbox{ and } \lim \int_a^b u_n(x) dG(x) = \int_a^b u(x) dG(x)\;.$$

From Step 2 and Step 3, we have that if $F\succeq_{\alpha-DCX}G$ then for every decreasing, twice-differentiable, and convex function $u:[a,b]\to \R_+$ we have $\int_a^b u(x) dF(x) \ge \int_a^b u(x)  dG(x)$. 
\end{proof}

\begin{proof}[Proof of Proposition \ref{prop: maximal of sufficient}]
Assume that $\int_a^b u(x) dF(x) \geq \int_a^b u(x) dG(x)$ for all $u \in  AP_{2,[a,b]} + \mathcal{I}_{2,[a,b]}$.

Because $AP_{2,[a,b]}$ and $\mathcal{I}_{2,[a,b]}$ contain the zero function we have $\mathcal{I}_{2,[a,b]} \subseteq AP_{2,[a,b]} + \mathcal{I}_{2,[a,b]}$ and $AP_{2,[a,b]} \subseteq AP_{2,[a,b]} + \mathcal{I}_{2,[a,b]}$. Hence, the $2,[a,b]$-concave function $u = - \max\{c-x,0\}^{2}$ belongs to $AP_{2,[a,b]} + \mathcal{I}_{2,[a,b]}$. Thus, 
 $$  \int_a^b \max\{c-x,0\}^{2} dF(x) \leq  \int_a^b \max\{c-x,0\}^{2} dG(x).$$
 From the proof of Proposition \ref{prop:4} it is enough to show that 
 $$  \int_a^b \max\{c-x,0\} (b-x) dF(x) \leq \int_a^b \max\{c-x,0\} (b-x) dG(x)$$ 
 in order to prove that $F \succeq_{2,[a,b]-S} G$. Let $c \in [a,b] $. Integration by parts (see Lemma \ref{lemma:0}) implies
 \begin{align*}
     \int_a^b \max\{c-x,0\} (b-x) dF(x) & =  \int_a^c (c-x) (b-x) dF(x) \\
     & =  - \int_a^c (-c-b+2x)F(x) dx  \\
     & =   \int_a^c (c-x)F(x) dx  + \int_a^c (b-x)F(x) dx \\  
     & = \frac{1}{2} \int_a^b  \max\{c-x,0\}^{2} dF(x) +\frac{1}{2} \int_a^c F(x) dk(x)   
 \end{align*}
where $k(x):=-(b-x)^{2}$. 

Note that $k$ is strictly increasing on $[a,b]$ and $-k''(x)/k'(x)=1/(b-x)$. Hence, 
from Theorem 2 in \cite{meyer1977second} the fact that $\int_a^b u(x) dF(x) \geq \int_a^b u(x) dG(x)$ for all $u \in AP_{2,[a,b]}$ implies that $\int_a^c F(x) dk(x) \leq \int_a^c G(x) dk(x)$. Thus, 
 \begin{align*}
     \int_a^b \max\{c-x,0\} (b-x) dF(x) 
     & = \frac{1}{2} \int_a^b  \max\{c-x,0\}^{2} dF(x) +\frac{1}{2} \int_a^c F(x) dk(x) \\
     & \leq  \frac{1}{2} \int_a^b  \max\{c-x,0\}^{2} d G(x)  +\frac{1}{2} \int_a^c G(x) dk(x) \\
     & = \int_a^b \max\{c-x,0\} (b-x) dG(x). 
 \end{align*}
We conclude that
 $F \succeq_{2,[a,b]-S} G$. 
\end{proof}

\bibliographystyle{ecta}
\bibliography{alphaconvex}

\end{document}